\documentclass[twoside,a4paper,12pt,reqno]{amsart}

\usepackage{bbm}
\usepackage{amssymb}
\usepackage{tikz}
\usetikzlibrary{matrix, arrows}
\usepackage{fullpage}

\usepackage[pdftex,colorlinks,hypertexnames=false,linkcolor=blue,urlcolor=blue,citecolor=blue]{hyperref}

\newtheorem{Theorem}{Theorem}[section]
\newtheorem{Lemma}[Theorem]{Lemma}
\newtheorem{Proposition}[Theorem]{Proposition}
\newtheorem{Corollary}[Theorem]{Corollary}

\theoremstyle{definition}
\newtheorem{Definition}[Theorem]{Definition}

\theoremstyle{remark}
\newtheorem{Remark}[Theorem]{Remark}

\numberwithin{equation}{section}

\def\lmod{\!\operatorname{-mod}}

\def\ad{\operatorname{ad}}
\def\Ann{{\operatorname{Ann}}}
\def\coker{{\operatorname{coker}}}
\def\End{{\operatorname{End}}}
\def\gr{\operatorname{gr}}
\def\Hom{\operatorname{Hom}}
\def\Lie{\operatorname{Lie}}
\def\Mat{{\operatorname{Mat}}}
\def\op{{\operatorname{op}}}
\def\pr{{\operatorname{pr}}}
\def\Spec{\operatorname{Spec}}
\def\tw{{\operatorname{tw}}}

\def\sub{\subseteq}
\def\into{\hookrightarrow}
\def\onto{\twoheadrightarrow}
\def\isoto{\overset{\sim}{\longrightarrow}}

\def\hJ{{\widehat{J}}}
\def\hK{{\widehat{K}}}
\def\hU{{\widehat{U}}}
\def\hZ{{\widehat{Z}}}
\def\hTheta{{\widehat{\Theta}}}

\def\C{{\mathbb C}}
\def\R{{\mathbb R}}
\def\Z{{\mathbb Z}}
\def\Q{{\mathbb Q}}

\def\kk{{\mathbbm k}}

\def\a{\mathfrak a}
\def\g{\mathfrak{g}}
\def\h{\mathfrak h}

\def\l{\mathfrak{l}}
\def\m{\mathfrak m}
\def\n{\mathfrak n}
\def\v{\mathfrak{v}}
\def\sl{\mathfrak{sl}}

\def\bfp{\bar{\mathfrak p}}
\def\bfa{\bar{\mathfrak a}}

\def\cB{\mathcal B}
\def\cD{\mathcal D}
\def\cN{\mathcal N}
\def\cO{\mathcal O}
\def\cP{\mathcal P}

\def\GL{\mathrm{GL}}
\def\SL{\mathrm{SL}}

\def\ba{\text{\boldmath$a$}}
\def\bb{\text{\boldmath$b$}}
\def\bc{\text{\boldmath$c$}}
\def\bv{\text{\boldmath$v$}}

\title
[Modular finite $W$-algebras]
{\boldmath
Modular finite $W$-algebras}
\author{Simon M.~Goodwin and Lewis W.~Topley}
\address{School of Mathematics,
University of Birmingham,
Birmingham, B15 2TT,
UK}
\email{s.m.goodwin@bham.ac.uk}
\address{School of Mathematics, Statistics and Actuarial Science, University of Kent,
Canterbury, Kent CT2 7NF, UK}
\email{L.Topley@kent.ac.uk}


\begin{document}

\begin{abstract}
Let $\kk$ be an algebraically closed field of characteristic $p > 0$ and let
$G$ be a connected reductive algebraic group over $\kk$.  Under some standard hypothesis
on $G$, we give a direct approach to the finite $W$-algebra $U(\g,e)$ associated to
a nilpotent element $e \in \g = \Lie G$.  We prove a PBW theorem and
deduce a number of consequences, then move on to define and study the $p$-centre of $U(\g,e)$, which allows us to
define reduced finite $W$-algebras $U_\eta(\g,e)$ and we verify that they
coincide with those previously appearing in the work of Premet.
Finally, we prove a modular version of Skryabin's equivalence of categories,
generalizing recent work of the second author.
\end{abstract}

\maketitle

\section{Introduction}

Finite $W$-algebras were introduced into the mathematical literature by Premet in \cite{PrST},
and have subsequently become an area of great interest.  Given a nilpotent element $e$ of the Lie algebra $\g_\C$ of
a complex connected reductive algebraic group $G_\C$, one can associate the finite $W$-algebra $U(\g_\C,e)$.  This is a certain
associative algebra, which can be viewed as a quantization of the Slodowy slice through the adjoint orbit of $e$.
These $W$-algebras and their representation theory have been extensively
studied and have found many important applications, a notable example being the theory of primitive
ideals of the universal enveloping algebra of $\g_\C$.  Many of the applications stem from Skryabin's equivalence,
which is proved in \cite{Sk}.  We refer to the survey article
by Losev \cite{Losev} for an overview.

Let $\kk$ be an algebraically closed field of characteristic $p > 0$.
We work under some mild assumptions on $G$ as given in \S\ref{ss:redgrps},
which, in particular, prescribe the existence of a non-degenerate symmetric $G$-invariant bilinear form $(\cdot\,,\cdot)$ on
$\g$, and we define $\chi : = (e,\cdot) \in \g^*$.
The modular versions of finite $W$-algebras were first studied by Premet in
\cite{PrST};
in the current paper we denote those algebras as $U_\chi(\g,e)$ and refer
to them as restricted finite $W$-algebras.

We recall that the centre of $U(\g)$ contains the $p$-centre $Z_p(\g)$; the definition
of $Z_p(\g)$ is given in \S\ref{ss:notation}.  Since $U(\g)$ is a finite module over $Z_p(\g)$, and the latter acts by scalars on irreducible representations
we see that every irreducible representation of $U(\g)$ is finite dimensional, and factors through
a reduced enveloping algebra $U_\eta(\g)$ for some $\eta \in \g^*$.
The reduced enveloping algebra
$U_\eta(\g)$ is obtained by quotienting out by the character of $Z_p(\g)$ determined by $\eta$
as recalled in \S\ref{ss:notation}.

Premet proved in \cite[Theorem~2.4]{PrST} that the reduced enveloping algebra $U_\chi(\g)$ is a matrix algebra
of degree $p^{d_\chi}$ over $U_\chi(\g,e)$, where $d_\chi$ is half the dimension
of the coadjoint $G$-orbit of $\chi$.  This leads to an equivalence of categories
between $U_\chi(\g)\lmod$ and $U_\chi(\g,e)\lmod$, which we refer to as Premet's equivalence.
In particular, this gives an alternative proof of the Kac--Weisfeiler conjecture that
the dimension of any (finite dimensional) $U_\chi(\g)$-module is divisible for $p^{d_\chi}$.
This was previously proved in \cite{PrKW}, and we also note that there
is a reduction, so that the assumption that $\chi$ is nilpotent can be removed.

The representation theory of reduced enveloping algebras $U_\eta(\g)$ attracted a great deal of research interest
from leading mathematicians including Friedlander--Parshall, Humphreys,
Jantzen, Kac and Premet in the late 20th century, we refer to the survey articles
\cite{JaLA} and \cite{Hu} for an overview.  As mentioned above, there is a reduction
to the case $\eta = \chi = (e,\cdot)$ for nilpotent $e \in \g$ is given in
\cite{FP}, see also \cite{KW}.
A notable advance in the field concerns a conjecture made by Lusztig \cite{Lu} predicting a deep connection between
the representation theory of $U_\chi(\g)$ and the geometry
of the corresponding Springer fibre $\cB^\chi$ over $\C$.
This conjecture was proved for $p$ sufficiently large by Bezrukavnikov--Mirkovic in 2010 \cite{BM}
building on the work by Bezrukavnikov--Mirkovic--Rumynin \cite{BMR} that
relates $U_\chi(\g)$ with $D$-modules on $\cB^\chi$.

The use of $W$-algebras to study the representation theory of the reduced enveloping algebras
has led to significant progress.  As already mentioned Premet's equivalence
gives an alternative proof of the Kac--Weisfeiler conjecture.
Another famous conjecture in the representation theory of $U_\chi(\g)$,
often known as Humphreys' conjecture, states that there is actually a $U_\chi(\g)$-module
of dimension $p^{d_\chi}$. Through Premet's equivalence this is equivalent to the existence of a $1$-dimensional
module for $U_\chi(\g,e)$.
As explained in the introduction to \cite{PrMF} Humphreys' conjecture is now a theorem for $p$ larger than an unspecified bound,
and Premet states that in his future work this bound will be lifted.  The proof is dependent on the
existence of $1$-dimensional modules for $U(\g_\C,e)$, the final cases of which are resolved in \cite{PrMF}; we refer to
\cite[Section~7]{Losev} or the introduction to \cite{PrMF} for more on the existence of $1$-dimensional $U(\g_\C,e)$-modules,
and to \cite{PT} for another recent development.

In this paper, which is inspired by the work of Premet, we give a new approach to the modular versions of $W$-algebras
through an algebra $U(\g,e)$ given in Definition~\ref{D:walgebra}.  Informally, this algebra
can be viewed as the characteristic $p$ version of $U(\g_\C,e)$.  Moreover,
we define the $p$-centre $Z_p(\g,e)$ of $U(\g,e)$ and then reduced finite $W$-algebras
$U_\eta(\g,e)$ are defined by quotienting out by a character of $Z_p(\g,e)$.
This is done so that $U_\eta(\g,e)$ bears the same relationship to $U(\g,e)$ as the reduced enveloping algebra $U_\eta(\g)$ bears to $U(\g)$.
For $p$ sufficiently large, the algebra $U(\g,e)$ has appeared in the work of Premet in \cite{PrPI} and
\cite{PrCQ}, where it is obtained from $U(\g_\C,e)$
through reduction modulo $p$.
In future work, we aim to clarify the relationship between $U(\g_\C,e)$ and $U(\g,e)$ further,
especially so that it is applicable for small $p$.  Many of our results
offer an alternative perspective on results of Premet from \cite{PrST}, \cite{PrPI}, \cite{PrCQ} and \cite{PrGR}, and
are applicable for small $p$.

We view the main results of this paper as foundations to further exploit this relationship
between the representation theory of $U(\g,e)$ and that of $U_\chi(\g)$.
In forthcoming work, we use the theory developed in this paper to classify
all $U_\chi(\g)$-modules of dimension $p^{d_\chi}$ when $G = \GL_n(\kk)$.  Further,
we show that all of these minimal dimensional modules can be parabolically induced in an appropriate
sense, which gives a modular analogue of M{\oe}glin's theorem, see \cite{Moe}.  The approach is similar
to that given by Brundan in \cite{BrM}, where an alternative proof of M{\oe}glin's theorem is presented.
We expect that the detailed results on the representation theory of $U(\g_\C,e)$, for $\g_\C = \mathfrak{gl}_n(\C)$ obtained by
Brundan--Kleshchev in \cite{BKrep} will lead to further applications to representation theory
of reduced enveloping algebras in type $\mathrm A$.

We now outline the contents of this paper, and highlight the main results.

In Section~\ref{S:prelims}, we set the stage by recalling the notation and elementary properties of
reductive groups and their Lie algebras, and we state the results from the
theory of nilpotent orbits that we require. In Section~\ref{S:goodgrading} we
explain how to adapt the theory of good gradings from \cite{EK} and \cite{BruG} to the modular
setting.

We move on to the main body of the paper in Section~\ref{S:walgebra}, in which $U(\g,e)$ is defined.
We fix an integral good grading $\g = \bigoplus_{j \in \Z} \g(j)$ of $\g$, and use this to define a nilpotent
subalgebra $\m = \l \oplus \bigoplus_{j < -1} \g(j)$, where $\l$ is chosen to be a Lagrangian subspace of $\g(-1)$ with
respect to the symplectic form on $\g(-1)$ defined in \S\ref{ss:compatiblespaces}.  Then
we let $M$ be the connected unipotent subgroup of $G$ with Lie algebra $\m$.  We define $\m_\chi := \{x - \chi(x) \mid x \in \m\} \sub U(\g)$
and $I = U(\g)\m_\chi$ to be the left ideal of $U(\g)$ generated by $\m_\chi$.  We note that the adjoint action
of $M$ on $U(\g)$ induces a well-defined adjoint action of $M$ on $U(\g)/I$.
In Definition~\ref{D:walgebra} the {\em finite $W$-algebra associated to $e$} is defined to be
$$
U(\g,e) := \{u + I \in U(\g)/I \mid g \cdot u + I = u + I \text{ for all } g \in M\}.
$$
This differs from the usual definition of $U(\g_\C,e)$ in that we take invariants with respect to a group
action rather than a Lie algebra.
However, we note that this is the correct characteristic $p$
analogue of $U(\g_\C,e)$ as justified by the PBW theorem for $U(\g,e)$ given in Theorem~\ref{T:PBWtheorem}.
This PBW theorem asserts that $U(\g,e)$ is a filtered deformation of the coordinate algebra of a so called
{\em good transverse slice} $e + \v$ to the nilpotent orbit of $e$ as was introduced by Spaltenstein
in \cite{Sp}, see \S\ref{ss:goodtransverse} for the definition of $e + \v$.
We note that in characteristic zero, we can take our good transverse slice to be the
Slodowy slice $e+\g_\C^f$, where $(e,h,f)$ is an $\sl_2$-triple in $\g_\C$, and $\g_\C^f$ denotes the
centralizer of $f$ in $\g_\C$.  However, when working over $\kk$, though
our assumptions on $\g$ do ensure existence of an $\sl_2$-triple, the Slodowy
slice can fail to be transversal to the orbit of $e$.
Our proof of Theorem~\ref{T:PBWtheorem} follows similar lines to the proof of
the corresponding theorem for $U(\g_\C,e)$ given in \cite{GG}, however we have
to work with cohomology for the algebraic group $M$, and make a number
of adaptations.  We note that as is the case in characteristic zero, we can think of $U(\g,e)$ being obtained
by quantum Hamiltonian reduction, see for example \cite{GG}.
We also consider the {\em extended finite $W$-algebra} $\hU(\g,e)$, which is defined
to be the $\m$-invariants in $U(\g)/I$ and we give a PBW theorem for $\hU(\g,e)$ in Theorem~\ref{T:extendedPBWtheorem}.

In Section~\ref{S:indep}, we prove that $U(\g,e)$ is independent up to isomorphism of the choice
of Lagrangian space $\l$ and the choice of good grading. This is achieved by adapting the proofs of these
results from \cite{GG} and \cite{BruG}, which is possible given the PBW theorem for $U(\g,e)$.
Thus we have that up to isomorphism $U(\g,e)$ only depends on the adjoint $G$-orbit of $e$.

We study the structure of $U(\g,e)$ and $\hU(\g,e)$ further in Section~\ref{S:applicationsofPBW}, where
we interpret the PBW theorems for $U(\g,e)$ and $\hU(\g,e)$ more explicitly.  Here we follow the approach in \cite[\S3.2]{BGK},
which in turn is based on results in \cite{PrST} and \cite{PrJI}.  In particular,
we obtain the form of PBW generators in Theorem~\ref{T:PBWbasisthm}: this will be of importance
when considering techniques of reduction modulo $p$.  Also in Corollary~\ref{C:UvsUhat} we clarify the
relationship between $U(\g,e)$ and $\hU(\g,e)$, which for $p$ sufficiently large is contained in \cite[Theorem~2.1]{PrCQ}.
This corollary also shows that $U(\g,e)$ can be obtained as a quotient of $\hU(\g,e)$.

The $p$-centre of $U(\g,e)$ is introduced in Section~\ref{S:pcentre}.  This is defined similarly to $U(\g,e)$ by letting
$I_p = I \cap Z_p(\g)$ and setting
$$
Z_p(\g,e) := \{ u + I_p \in Z_p(\g)/I_p \mid g \cdot u + I_p = u+I_p \text{ for all } g \in M\}.
$$
We also note that $Z_p(\g)/I_p$ can be viewed as the $p$-centre of $\hU(\g,e)$.
Our main result about the $p$-centres of $U(\g,e)$ and $\hU(\g,e)$ is Theorem~\ref{T:extendedpcentre}, where we show that
both $U(\g,e)$ and $\hU(\g,e)$ are free of rank $p^{\dim \g^e}$ over their $p$-centres.
Using the PBW theorem for $U(\g,e)$ we can show that the maximal ideals of $Z_p(\g,e)$
are parameterized by the affine subspace $\chi + \check\v$ of $\g^*$, which is dual to $e+\v$ via $(\cdot\,,\cdot)$.
We define the reduced finite $W$-algebras associated to $\eta \in \chi + \check\v$ to be
$$
U_\eta(\g,e) := U(\g,e)/K_\eta U(\g,e),
$$
where $K_\eta$ is the maximal ideal of $Z_p(\g,e)$ corresponding to $\eta$;
the restricted $W$-algebra $U_\chi(\g,e)$ is obtained in the case $\eta = \chi$.
An important result
for us is Proposition~\ref{P:reducesame}, which says that $U_\eta(\g,e)$ coincides (up to isomorphism)
with the reduced finite $W$-algebras introduced by Premet in \cite{PrST}.

In the final section of this paper, we state and prove a modular analogue of Skryabin's equivalence
from \cite{Sk}.  As in the characteristic zero case, this gives an equivalence of categories between
$U(\g,e)\lmod$ and a certain category of modules for $U(\g)$.  We expect that this result will be of importance
in clarifying the relationship between the representation theory of $U(\g,e)$ and that of $U(\g)$.
Our approach develops that of the second author \cite{To} where
a similar result was proved for $p$ sufficiently large.  As mentioned in Remark~\ref{R:genequiv},
this leads to an alternative approach to Premet's equivalence.

\smallskip

\noindent {\bf Ackowledgements:} Both authors would like to thank the University of Padova and the Erwin
Schr{\"o}dinger Institute, Vienna, where parts of this work were carried out. The second author would like to offer
special thanks to Giovanna Carnovale and Francesco Esposito for useful conversations over the past two years;
furthermore the research leading to these results received funding from the European Commission, Seventh
Framework Programme, Grant Agreement 600376, as well as grants CPDA125818/12, 60A01-4222/15 and DOR1691049/16
from the University of Padova.

We thank the referees for a number of helpful comments, which have improved the exposition of the paper.

\section{Preliminaries} \label{S:prelims}

\subsection{Notation for linear algebraic groups and Lie algebras} \label{ss:notation}
Let $p \in \Z_{>0}$ be a prime number and
let $\kk$ be an algebraically closed field of characteristic $p$.

Throughout this paper, we write $V^{(1)}$ for the {\em Frobenius twist} of a vector space $V$ over $\kk$,
so $V^{(1)}$ is equal to $V$ as an abelian group, but the scalar multiplication is determined by saying
$a \in \kk$ acts on $V^{(1)}$ as $a^{\frac{1}{p}}$ acts on $V$.

Let $G$ be a linear algebraic group over
$\kk$, and write $\g = \Lie G$ for the Lie algebra of $G$.
We write $G^\circ$ for the identity component of $G$, the derived subgroup of $G$ is denoted
$\cD G$ and we write $Z(G)$ for the centre of $G$.

Let $g \in G$ and $x \in \g$.  We write $g \cdot x$ for the image of $x$ under $g$ in the adjoint action,
$G^x$ for the centralizer of $x$ in $G$ and $\g^x$ for the centralizer of $x$ in $\g$.  Also we
define $G^{[x]} := \{g \in G \mid g \cdot x \in \kk x\}$ and call it the {\em normalizer of $x$ in $G$}.
We use similar notation when considering the coadjoint action of $G$ on $\g^*$.

Since $\g$ is the Lie algebra of an algebraic group, it is a restricted Lie algebra in a natural way, and we write $x \mapsto x^{[p]}$ for the
$p$th power map. The elements $x^p - x^{[p]}$ for $x\in \g$ are central in the enveloping algebra $U(\g)$
and the algebra $Z_p(\g)$ that they generate is known as the {\em $p$-centre} of $U(\g)$. It is an immediate consequence of the PBW theorem that
$U(\g)$ is a free $Z_p(\g)$-module of rank $p^{\dim \g}$.
The adjoint action of $G$ on $\g$ extends to an action on $U(\g)$ by automorphisms, and this action preserves
$Z_p(\g)$. Similarly $S(\g)$ admits a adjoint $G$-action by automorphisms. Furthermore, there is a $G$-equivariant
$k$-algebra isomorphism
\begin{equation}\label{e:ximap}
 \xi : S(\g)^{(1)} \isoto Z_p(\g)
\end{equation}
determined by sending $x \mapsto x^p - x^{[p]}$ for $x\in \g$.

Let $\eta \in \g^*$.  We define $J_\eta$ to be the ideal of $Z_p(\g)$  generated by $\{x^p - x^{[p]} - \eta(x)^p \mid x \in \g\}$
and the {\em reduced enveloping algebra} $U_\eta(\g) := U(\g)/J_\eta U(\g)$.  Since $U(\g)$ is a free $Z_p(\g)$-module of rank $p^{\dim \g}$ we observe that
$\dim U_\eta(\g) = p^{\dim \g}$.  In the case $\eta = 0$, we have that $U_0(\g)$ is the {\em restricted enveloping algebra of $\g$}.

\subsection{Reductive groups} \label{ss:redgrps}
Throughout the rest of this article we let $G$ be a connected reductive algebraic group over $\kk$.
Also we assume that $G$ satisfies the standard hypotheses from \cite[\S6.3]{JaLA}:
\begin{itemize}
\item[(H1)] $\cD G$ is simply connected;
\item[(H2)] $p$ is a good prime for $G$; and
\item[(H3)] there is a non-degenerate $G$-invariant symmetric bilinear form on $\g$.
\end{itemize}
We refer the reader to \cite[\S6.4]{JaLA} for a discussion of what these hypotheses amount to,
and remark that they are stable under taking Levi subgroups.  Here we just mention that
for $G$ simple these hypotheses hold for $G$ under the following conditions: for type $\mathrm A_r$ we require $p$ does not divide
$r+1$; for types $\mathrm B_r$, $\mathrm C_r$ or $\mathrm D_r$ we require $p \neq 2$; for types $\mathrm G_2$, $\mathrm F_4$,
$\mathrm E_6$ or $\mathrm E_7$ we require $p \neq 2,3$;
and for type $\mathrm E_8$ we require $p \neq 2,3,5$.  Further, we note that although (H3)
does not hold for $\SL_n(\kk)$ if $p \mid n$, it does hold for $\GL_n(\kk)$.

We fix a form as in (H3) and denote it by $(\cdot\,,\cdot) : \g \times \g \to \kk$.
The existence of this form has two important consequences:
\begin{itemize}
\item[(i)] there exists an isomorphism of $G$-modules $\g \isoto \g^*$;
\item[(ii)] $\g^x = \Lie(G^x)$ for all $x \in \g$, see \cite[\S2.5]{JaNO}.
\end{itemize}
We use (i) to identify $\g \cong \g^*$ when it is convenient to do so.
We remark that (ii) is a nontrivial consequence of the existence of $(\cdot\,,\cdot)$
and can be expressed by saying that the scheme theoretic centralizer of $x \in \g$ is smooth.
Another piece of notation we introduce is that we write $\a^\perp := \{x \in \g \mid (x,y) = 0 \text{ for all } y \in \a\}$
for a subspace $\a$ of $\g$.

Let $T$ be a maximal torus in $G$
and let $\Phi$ be the root system of $G$ with respect to $T$.
Given $\alpha \in \Phi$,
we write $\g_\alpha$ for the root subspace of $\alpha$ in $\g$ and we fix a parametrization $u_\alpha : \kk \to G$
of the root subgroup $U_\alpha$ of $\alpha$,
Then $e_\alpha = du_\alpha(1)$ is a generator of $\g_\alpha$.

This can all be chosen so that the adjoint action of the root subgroups on the root spaces is
given by the following formulae.
For $\alpha, \beta \in \Phi$ with $\beta \neq \pm \alpha$,
we have
\begin{align} \label{e:adjoint}
u_\alpha(t) \cdot e_\alpha &= e_\alpha  \\
u_\alpha(t) \cdot e_{-\alpha} &= e_{-\alpha} +t h_\alpha - t^2 e_\alpha  \nonumber \\
u_\alpha(t) \cdot e_\beta &= e_\beta + \sum_{\substack{i > 0 : \\ i\alpha+\beta \in \Phi^+}} m_{\alpha,\beta,i} t^i e_{i\alpha+\beta}, \nonumber
\end{align}
where $h_\alpha = [e_\alpha,e_{-\alpha}]$ and the $m_{\alpha,\beta,i} \in \Z$ are certain structure constants
satisfying $[e_\alpha,e_\beta] = m_{\alpha,\beta,1} e_{\beta + \alpha}$;
also the structure constants $m_{\alpha,\beta,i}$ are always $\pm 1$, $\pm 2$ or $\pm 3$, and
$\pm 2$ occurs only if $G$ has a simple component of type $\mathrm B_r$, $\mathrm C_r$, $\mathrm F_4$
or $\mathrm G_2$, and $\pm 3$ occurs only if $G$ has a simple component of type $\mathrm G_2$.
These formulae can be obtained from Steinberg's construction of Chevalley groups in \cite{StL},
see also \cite[Chapter~4]{CaS}.

We write $X^*(T)$ for the group of characters of $T$ and $X_*(T)$ for the group of cocharacters of $T$; both of which we
write additively.  The perfect pairing between $X^*(T)$ and $X_*(T)$ is denoted by $\langle \cdot\,, \cdot \rangle :
X^*(T) \times X_*(T) \to \Z$.  For a subtorus $S$ of $T$, we write $X^*(S)$ for the corresponding quotient of $X^*(T)$ and $X_*(S)$ for the corresponding
subgroup of $X_*(T)$.

\begin{Remark} \label{R:char0a}
For the main part of this paper, we work entirely in characteristic $p$ (and always
under the hypothesis given in \S\ref{ss:redgrps}).   So we just consider $G$ as a reductive algebraic group over $\kk$,
and identify it with its group of points over $\kk$. However,
it is convenient for us to relate our work with
the characteristic zero situation.  In this paper this is of use when considering
good gradings in Section~\ref{S:goodgrading}, and
for the argument used in the proof of Theorem~\ref{T:extendedPBWtheorem}.  Also it will be of importance when considering
reduction modulo $p$ in future work. To enable us to do this, it
is convenient for us to consider a connected split reductive group scheme $G_\Z$ over $\Z$ such that $G = G_\Z(\kk)$.
We refer the reader to \cite[Part~2,~Section 1]{JaAG} for the notation and terminology used here.

In particular, we note that it is possible to choose $T = T_\Z(\kk)$, where
$T_\Z$ is a split maximal torus of $G_\Z$.  Then the homomorphisms $u_\alpha$ can be chosen
to be defined over $\Z$; as such they are uniquely determined up to sign, and the Frobenius
morphism $F : G \to G$ is given on $U_\alpha$ by $F(u_\alpha(t)) = u_\alpha(t^p)$.
Further, we can view $e_\alpha$ as an element of $\g_\Z = \Lie G_\Z$.
We refer to \cite[Part~2,~\S1.11]{JaAG} for more details of the constructions in this paragraph.
\end{Remark}

\subsection{Nilpotent orbits and restricted root systems} \label{ss:nilporbits}

Let $e \in \g$ be a nilpotent element.  We say that $e$ is {\em compatible with $T$} if
$(T^{[e]})^\circ$ is a maximal torus of $G^{[e]}$.  By choosing $T$ to contain a maximal torus of $G^{[e]}$,
we can assume that $e$ is compatible with $T$.
We write $T_e = (T^e \cap \cD G)^\circ$ and $T_{[e]} = (T^{[e]} \cap \cD G)^\circ$.

Let $G_0$ be the centralizer of $T_e$ in $G$ and $\g_0 = \Lie G_0$.
Then $G_0$ is a Levi subgroup of $G$ and we also note
that $\g_0$ is the set of $T_e$-fixed points in $\g$.  Moreover, $e$ is distinguished nilpotent in
$\g_0$; we recall that this means any torus in $G_0^e$ is central, and this implies
that $Z(G_0)^\circ = (T^e)^\circ$ is the unique maximal torus of $G_0^e$.

We recall that a cocharacter $\lambda : \kk^\times \to G$ is {\em associated to $e$} if
$\lambda(t) \cdot e = t^2 e$ and $\lambda(\kk^\times)$ is contained in the derived subgroup of
a Levi subgroup in which $e$ is distinguished, see \cite[Definition~5.3]{JaNO}.
By \cite[Lemma~5.3]{JaNO} and our choices, there exists a cocharacter
$\lambda \in X_*(T_{[e]})$ associated to $e$, and all associated cocharacters are conjugate under
$(G^e)^\circ$.  Thus we observe that $\lambda$ is the only cocharacter
in $X_*(T_{[e]})$ associated to $e$:
to see this we recall that $G^{e,\lambda} := \{x \in G^e \mid \lambda(t) \cdot x = x \text{ for all } t \in \kk^\times\}$
is a Levi factor of $G^e$, see \cite[Proposition~5.10]{JaNO}; and note that
if $\lambda$ is conjugate to $\mu \in X_*(T_{[e]})$ via $(G^e)^\circ$, then
it must actually be conjugate via $G^{e,\lambda}$, so that $\mu = \lambda$.

The grading
\begin{equation} \label{e:dynkingrading}
\g = \bigoplus_{j \in \Z} \g(j;\lambda)
\end{equation}
where $\g(j;\lambda) := \{x \in \g \mid \lambda(t) \cdot x = t^j x \text{ for all } t \in \kk^\times\}$ is
called the {\em Dynkin grading of $\g$}.
As explained in \cite[\S5.5]{JaNO} the Dynkin grading is the correct analogue of the grading in characteristic 0
given by the $\ad h$-eigenspace decomposition
coming from an $\mathfrak{sl}_2$-triple $(e,h,f)$.

We recall that $\g^e$ is contained in the non-negative part of the Dynkin grading, see
\cite[Proposition~5.8]{JaNO}, so we have
\begin{equation} \label{e:g^e}
\g^e = \bigoplus_{j \in \Z_{\ge 0}} \g^e(j;\lambda).
\end{equation}

We denote by $\Phi^e \sub X^*(T_e)$ the \emph{restricted root system}: the set of non-zero restrictions $\alpha|_{T_e}$ where
$\alpha \in \Phi$.  This is equivalent to the restricted root system defined in \cite[Section~2]{BruG}, so all of the theory
developed there applies.

For $\alpha \in X^*(T_e)$, we write $\g_\alpha := \{x \in \g \mid t \cdot x = \alpha(t) x \text{ for all } t \in T_e\}$
for the $T_e$-weight space corresponding to $\alpha$.  So $\g_\alpha$ is zero unless $\alpha \in \Phi^e \cup \{0\}$,
and we have
$$
\g = \g_0 \oplus \bigoplus_{\alpha \in \Phi^e} \g_\alpha.
$$
Incorporating the Dynkin grading, we obtain the decompositions
$$
\g = \bigoplus_{j \in \Z} \g_0(j;\lambda) \oplus \bigoplus_{\substack{\alpha \in \Phi^e \\ j \in \Z}} \g_\alpha(j;\lambda),
$$
where $\g_\alpha(j;\lambda) : = \g_\alpha \cap \g(j;\lambda)$, and
\begin{equation} \label{e:g^eroot}
\g^e = \bigoplus_{j \in \Z_{\ge 0}} \g_0^e(j;\lambda) \oplus \bigoplus_{\substack{\alpha \in \Phi^e \\ j \in \Z_{\ge 0}}} \g^e_\alpha(j;\lambda),
\end{equation}
where $\g^e_\alpha(j;\lambda):= \g_\alpha(j;\lambda) \cap \g^e$.
Since $\g^e$ is contained in the non-negative part of the Dynkin grading we deduce
that
\begin{equation} \label{e:dimg^e}
\dim \g^e_\alpha(j;\lambda) = \dim \g_\alpha(j;\lambda) - \dim \g_\alpha(j+2;\lambda)
\end{equation}
for $\alpha \in \Phi^e \cup \{0\}$ and $j \in \Z_{\ge 0}$.

\begin{Remark} \label{R:char0}
We continue the discussion from Remark~\ref{R:char0a}, and use the notation from there.
In particular we recall that $\g_\Z$ is a $\Z$-form of $\g$, so that $\g \cong \g_\Z \otimes \kk$.
We fix a $\Z$-basis $\tilde \Pi$
of $X_*(T)$, and for $\beta^\vee \in \tilde \Pi$ we let $h_\beta = d\beta^\vee(1)$.
Then $\{e_\alpha \mid \alpha \in \Phi\} \cup \{h_\beta \mid \beta^\vee \in \tilde \Pi\}$
is a $\Z$-basis of $\g_\Z$.
We let $G_\C = G_\Z(\C)$ and $\g_\C = \Lie G_\C \cong \g_\Z \otimes \C$.
By an abuse of notation, we view $e_\alpha$ and $h_\beta$ both as elements
of $\g_\C$ and $\g$.

We note that in all of the following discussion, there is reduction to
the case where $\cD G$ is simple using standard arguments, so we may as well
assume that this is the case.

We write $\cN \sub \g$ for the variety of nilpotent elements in $\g$ and define $\cN_\C \sub \g_\C$
similarly.   By the Bala--Carter theory, see for example \cite[Chapter~4]{JaNO}, it is known that the nilpotent
orbits of $G$ in $\cN$ are
in canonical bijection with the nilpotent orbits of $G_\C$ in $\cN_\C$.
Moreover, for each nilpotent orbit $\cO_\C$ in $\cN_\C$, we can choose a representative
$e = \sum_{\alpha \in A} e_\alpha \in \g_\Z$, where $A \sub \Phi$;
moreover, this gives a representative of the corresponding orbit in $\g$.
For classical groups, this can be seen using the theory of Dynkin pyramids from \cite[Sections~4--7]{EK},
and for exceptional groups this can be observed from the tables in \cite[Section~11]{LT}.  By a minor
abuse of notation we view $e$ as both an element of $\g_\C$ and of $\g$.

It can be checked that this choice of $e$ is compatible with $T$, we recall that this means that $(T^{[e]})^\circ$
is a maximal torus of $G^{[e]}$; and also that $e \in \g_\C$ is compatible with $T_\C$.
Then it is a straightforward calculation in the root datum to find a $\Z$-basis
for $X_*(T^e)$, which is valid both over $\kk$ and over $\C$.
Indeed for the exceptional groups this is implicit in \cite[Section~11]{LT},
and for classical groups we can observe that $X_*(T^e)$ is determined by the Dynkin pyramid for $e$.
Moreover, the associated cocharacter $\lambda \in X_*(T^{[e]})$ is given by an expression
valid both for $\kk$ and for $\C$.

An outcome of all of this is that the dimensions of the spaces
$\g_\alpha(j;\lambda)$ for $\alpha \in \Phi^e \cup \{0\}$ and $j \in \Z$ in
the decomposition \eqref{e:g^eroot} are the same in characteristic $p$
as in characteristic 0.  Therefore, using \eqref{e:dimg^e}, we
see that the dimensions of the spaces of $\g^e_\alpha(j;\lambda)$ is the same in characteristic $p$
as in characteristic 0.  In particular, this implies that
\begin{equation} \label{e:dimgeweight}
\dim \g^e_{- \alpha}(j;\lambda) = \dim \g^e_\alpha(j;\lambda),
\end{equation}
as this is the case in characteristic 0, where it can be observed as a consequence
of the representation theory of $\sl_2$ as in the proof of \cite[Lemma~13]{BruG}.
It would be interesting to know if these equalities in dimensions in \eqref{e:dimgeweight} can
be proved in a more direct way.
\end{Remark}

\section{Good gradings} \label{S:goodgrading}
The theory of good gradings for nilpotent elements in $\g_\C$
has been well studied, see \cite{EK} and \cite{BruG},
and they can be used in the definition of $W$-algebras.
In this section we explain how to adapt this theory to the modular case, and show that
the parametrization of good gradings is essentially the same over $\kk$ as it is over $\C$.

For each $\gamma \in X_*(T_{[e]}) \setminus X_*(T_e)$,
there exists $d= \deg_\gamma(e)\in \Z \setminus \{0\}$ such that  $\gamma(t) \cdot e = t^d e$ for all $t \in \kk^\times$.
For $j \in \frac{2}{d}\Z$, we set
$$
\g(j;\gamma,e) := \{x \in \g \mid \gamma(t) x = t^{\frac{jd}{2}} x \text{ for all } t \in \kk^\times\}
$$
and observe that
$$
\g = \bigoplus_{j \in \frac{2}{d}\Z} \g(j;\gamma,e)
$$
is a grading of $\g$; we denote this grading by $\Gamma_\gamma$.  We remark that the normalization in the grading is chosen so that
$e \in \g(2;\gamma,e)$.  We say that the grading $\Gamma_\gamma$ is {\em integral} if $\g(j;\gamma,e) = 0$ for all $j \not\in \Z$, and we say that this grading
is {\em even} if $\g(j;\gamma,e) = 0$ for all $j \not\in 2\Z$.
Let $\sim$ be the equivalence relation on $X_*(T_{[e]})$ generated by $\gamma \sim m \gamma$ for $m \in \Z_{\neq 0}$
and denote the equivalence class of $\gamma$ by $[\gamma]$.
Since $T_{[e]} \sub \cD G$ and $(\cD G \cap Z(G))^\circ = 1$, we observe that $\Gamma_\gamma = \Gamma_{\gamma'}$ if and only if
$[\gamma] = [\gamma']$, for $\gamma, \gamma' \in X_*(T_{[e]}) \setminus X_*(T_e)$.

Now standard arguments can be used to prove that
for $i,j \in \frac{2}{d}\Z$, we have
\begin{itemize}
\item
if $x \in \g(j;\gamma,e)$, then $x^{[p]} \in \g(pj; \gamma,e)$; and
\item
$(\cdot\,, \cdot)$ restricts to a non-degenerate pairing $\g(j;\gamma,e) \times \g(-j;\gamma,e) \to \kk$
and $\g(i;\gamma,e)$ is orthogonal to $\g(j;\gamma,e)$ for $i \neq -j$.
\end{itemize}

We move on to the definition of good gradings used in this paper.

\begin{Definition}
Let $\gamma \in X_*(T_{[e]}) \setminus X_*(T_e)$.  We say that $\gamma$ is a \emph{good cocharacter for $e$} if
the map
$$
\ad e : \g(j;\gamma,e) \to \g(j+2;\gamma,e)
$$
is injective for all $j \leq -1$.
For a good cocharacter $\gamma$, we refer to the grading $\Gamma_\gamma$ as a {\em good grading for $e$}.

By convention the trivial cocharacter $\gamma$ is the only good cocharacter for $e =0$, and
$\deg_\gamma(0) = 2$.  In this case the grading $\Gamma_\gamma$ is concentrated in degree $0$.
\end{Definition}

It follows from \eqref{e:g^e} that the associated cocharacter $\lambda$ is a good cocharacter for
$e$, and that the grading $\Gamma_\lambda$ is integral.

We note that the proof of \cite[Theorem~1.3]{EK} can be adapted to show that
$\gamma$ is good for $e$ if and only if
$\ad e : \g(j;\gamma,e) \to \g(j+2;\gamma,e)$ is surjective for $j \ge -1$.

It is worth mentioning that we can consider more general good gradings for $e$, which are defined
by cocharacters with image in $G^{[e]}$.  However, standard arguments can be used to prove that any such good grading
is conjugate under $(G^e)^\circ$ to one defined by a cocharacter in $X_*(T^{[e]})$.  Then using
the fact that $X_*(T \cap \cD G) \oplus X_*(Z(G))$ has finite index in $X_*(T)$, we deduce that the grading can be
defined by a cocharacter in $X_*(T_{[e]})$.

We move on to describe the parametrization of equivalence classes of good cocharacters for $e$. By the above remarks
this is equivalent to the parametrization of the good gradings. Observe that there is a natural bijection
$$
\{ [\gamma] \mid \gamma \in X_*(T_{[e]}) \setminus X_*(T_e)\}\longrightarrow  E_e^\Q := \Q \otimes_\Z X_*(T_e)
$$
defined by $[\gamma] \mapsto 1\otimes \lambda - \frac{2}{d}\otimes \gamma$, where $d = \deg_\gamma(e)$; the inverse is given by the rule
$\frac{b}{c} \otimes \delta\mapsto [c\lambda - b\delta]$. Since every element of $E_e^\Q$ is of the form
$\frac{1}{c}\otimes \delta$ for some $c\in \Z$ and $\delta \in X_*(T_e)$ our goal is now to explain precisely when
$[c\lambda - \delta]$ gives a good grading.

We consider the decomposition of $\g^e$ given by \eqref{e:g^eroot}.
For $x \in \g^e_\alpha(j; \lambda)$, we see that the $c\lambda-\delta$ weight of $x$ is $cj-\langle \alpha,\delta \rangle$.
Thus for $c\lambda-\delta$ to be a good cocharacter we require
$cj-\langle \alpha,\delta \rangle > -c$ for all $\alpha \in \Phi^e$ and
all $j$ such that $\g^e_\alpha(j; \lambda) \neq \{0\}$.
If we set $m(\alpha) = 1 + \min\{j \ge 0 \mid \g^e_\alpha(j;\lambda) \neq \{0\}\}$ then it follows that $c\lambda-\delta$ is good
if and only if $c m(\alpha) > \langle \alpha, \delta \rangle$ for all $\alpha \in \Phi^e$. As explained at the end of
Remark~\ref{R:char0}, we have
$\dim \g^e_\alpha(j; \lambda) = \dim \g^e_{-\alpha}(j; \lambda)$ for $\alpha \in \Phi^e$ and $j \in \Z_{\ge 0}$. Thus the inequalities
corresponding to $\pm \alpha \in \Phi^e$ can be combined into the single inequality
$c m(\alpha) > |\langle \alpha, \delta \rangle|$.

Observe that the perfect pairing $X^*(T_e) \times X_*(T_e)  \rightarrow \Z$ induces a $\Z$-bilinear form $X^*(T_e) \times E_e^\Q \rightarrow \Q$.
We may now combine our previous remarks, to deduce that good gradings for $e$ are parameterized by a polytopal region
of $E_e^\Q$. More precisely:
\begin{enumerate}
\item[(i)] the gradings given by elements of $X_*(T_{[e]}) \setminus X_*(T_e)$ are parameterized by $E_e^\Q$;
\item[(ii)] the grading associated to $\frac{1}{c} \otimes \delta \in E_e^\Q$ is a good grading for $e$ if and only if
\begin{equation} \label{e:polytope}
|\langle \alpha,\tfrac{1}{c}\otimes \delta \rangle | < m(\alpha)
\end{equation}
for all $\alpha \in \Phi^e$.
\end{enumerate}
The polytope cut out by these inequalities is denoted $\cP_e^\Q$ and referred to as the {\em good grading polytope}.
Moreover, as a consequence of the discussion in Remark~\ref{R:char0},
we have that $m(\alpha)$ is the same as the corresponding value in characteristic 0.
Thus the good grading polytope $\cP_e^\Q$ is the intersection of $E_e^\Q$ with the polytope in $X_*(T_e) \otimes_\Z \R$ studied in \cite[Theorem~20]{BruG}.

\begin{Remark} \label{R:distinguished}
We briefly consider the case where $e$ is distinguished nilpotent in $\g$.  In this case we have that
$X_*(T_e) = \{0\}$, which implies that the only good grading for $\g$ is the Dynkin grading.
Moreover, the Dynkin grading is even, see \cite[\S5.7.6]{CaF}.

Moving back to the general case, we note that, by construction, $e \in \g_0$ is distinguished nilpotent.
Also any good grading of $\g$ for $e$ restricts to a good grading of $\g_0$ for $e$.  Therefore, we deduce
that $\g_0(j;\gamma,e) = 0$ for $j \not\in 2\Z$ for any good grading $\Gamma_\gamma$.
\end{Remark}

\section{\texorpdfstring{Modular $W$-algebras}{Modular W-algebras}} \label{S:walgebra}

In this section we move on to present the definitions of the $W$-algebras, which are the main
object of study in this paper. Throughout this section we fix a nilpotent element
$e \in \g$ that is compatible with $T$.  Let $\chi \in \g^*$
be defined by $\chi(x) = (e,x)$, where we recall that $(\cdot\,,\cdot)$ is an invariant non-degenerate symmetric
bilinear form on $\g$.  Then we have $\g^e = \g^\chi$, and let $d_\chi := \frac{1}{2} \dim G \cdot \chi = \frac{1}{2} \dim G \cdot e$.

We fix a good cocharacter $\gamma \in X_*(T_{[e]}) \setminus X_*(T_e)$ for $e$, where
we recall that $T_e = (T^e \cap \cD G)^\circ$ and $T_{[e]} = (T^{[e]} \cap \cD G)^\circ$.
Let $d= \deg_\gamma(e)\in \Z \setminus \{0\}$.  We abbreviate notation and write
$\g(j)$ as shorthand for $\g(j;\gamma,e)$, for $j \in \frac{2}{d}\Z$, and
we write
$$
\Gamma_\gamma : \g = \bigoplus_{j \in \frac{2}{d}\Z} \g(j)
$$
for the corresponding good grading.
By incorporating the restricted root
space decomposition, we obtain
\begin{equation} \label{e:goodrestricted}
\g = \bigoplus_{j \in \frac{2}{d}\Z} \g_0(j) \oplus \bigoplus_{\substack{\alpha \in \Phi^e \\ j \in \frac{2}{d}\Z}} \g_\alpha(j).
\end{equation}

\subsection{Compatible Lagrangian subspaces} \label{ss:compatiblespaces}

The skew-symmetric
bilinear form $\omega = \langle \cdot \,, \cdot \rangle$
is defined on $\g(-1)$  by
$$
\langle x, y \rangle := ([y,x],e) = \chi([y,x]);
$$
we note that $\omega$ is non-degenerate, because $\ad e : \g(-1) \to \g(1)$ is bijective and
$\g(-1)$ is dual to $\g(1)$ via $(\cdot \,, \cdot )$.

For any subspace $\l \sub \g(\-1)$ we write
$\l^{\perp_\omega} \sub \g(-1)$ for the annihilator of $\l$ with respect to $\omega$. We say that an
isotropic subspace $\l \sub \g(-1)$ is {\em compatible with $T$ and $\chi$} if $\l$ and $\l^{\perp_\omega}$
are $T$-stable under the adjoint action and $\chi(\l^{[p]}) = 0$; we note that
this second condition automatically holds if
$p \neq 2$.

An example of a compatible isotropic subspace of $\g(-1)$ is $\l = 0$.  We can also
obtain a compatible isotropic subspace of $\g(-1)$, which is in fact Lagrangian, by
choosing a system of positive roots $\Phi^e_+$ in $\Phi^e$ as defined in \cite[Section~2]{BruG}
and then setting $\l = \bigoplus_{\alpha \in \Phi^e_+} \g_\alpha(-1)$; here we require
Remark~\ref{R:distinguished} to see that $\l$ is Lagrangian.

We fix a compatible isotropic subspace $\l \sub \g(-1)$ for the rest of this section.

\subsection{\texorpdfstring{The Lie algebras $\m$ and $\n$ and the
group $N$}{The Lie algebras m and n and the group N}} \label{ss:algebraandgroup}

We define the
nilpotent subalgebras
$$
\m := \l + \bigoplus_{j < -1} \g(j) \quad \text{ and} \quad
\n := \l^{\perp_\omega} + \bigoplus_{j < -1} \g(j)
$$
of $\g$.
We note that $\m$ and $\n$ are
restricted subalgebras of $\g$, and they are stable under the adjoint
action of $T$.
There exists a subset $\Phi(\n) \sub \Phi$ with
$$
\n = \bigoplus_{\alpha \in \Phi(\n)} \g_\alpha.
$$
The one parameter subgroups $u_\alpha : \kk \to G$ associated to the roots $\alpha \in \Phi(\n)$ generate
a closed connected group $N \sub G$ defined over $\mathbb F_p$, stable under conjugation by $T$ and such that $\n = \Lie N$.

We have that $\chi$
defines a $1$-dimensional representation of $U(\m)$ and of $U_\chi(\m)$;
we denote the corresponding module by $\kk_\chi$.
The subspace $\m_\chi := \{x -\chi(x) \mid x\in \m\} \sub U(\g)$ is actually
a Lie subalgebra of $U(\g)$ isomorphic to $\m$.  Moreover, $\m_\chi$ generates the
kernel of $\chi : U_\chi(\m) \to \kk$ and this is the unique maximal ideal of
$U_\chi(\m) $, so that $U_\chi(\m)$ is a local algebra.

We observe that $\dim \g - \dim \m - \dim \n = \sum_{-1 \le j < 1} \dim \g(j)$.
Moreover, since
$$
\ad e :  \bigoplus_{j \ge -1} \g(j) \to \bigoplus_{j \ge 1} \g(j)
$$
is surjective and $\g^e \sub \bigoplus_{j \ge -1} \g(j)$
we have $\dim \g^e = \sum_{-1 \le j < 1} \dim \g(j)$.
Hence,
\begin{equation} \label{e:dimm+dimn}
\dim \m + \dim \n = \dim \g - \dim \g^e = 2d_\chi,
\end{equation}
where we recall that $d_\chi = \frac{1}{2} \dim G \cdot e$.

\subsection{\texorpdfstring{The definition of the finite $W$-algebra}
{The definition of the finite W-algebra}}
\label{ss:gelfandwhittaker}
We define the {\em generalized Gelfand--Graev module} to be
$$
Q := U(\g) \otimes_{U(\m)} \kk_\chi \cong U(\g)/U(\g)\m_\chi,
$$
where we recall that $\m_\chi = \{x -\chi(x) \mid  x\in \m\}$.  We
abbreviate notation and write $I$ for $U(\g)\m_\chi$.

\begin{Lemma} \label{L:Npreservesm}
Both $\m$ and $\m_\chi$ are stable under the adjoint actions of $\n$ and $N$.
\end{Lemma}

\begin{proof}
Observe that $[\n, \m] = [\n, \m_\chi] \sub \bigoplus_{i \le -2} \g(j) \sub \m$.
Moreover, if $x \in \l$ and $y \in \l^{\perp_\omega}$, then $\chi([y,x]) = 0$, so that in fact $[y,x] \in \m_\chi$, and it
follows that $[\n, \m] \sub \m_\chi$.

We can now use the formulae in \eqref{e:adjoint}, to see that if $g \in N$ and $x \in \m$, then
we have that $g \cdot x = x + y$, where $y \in \bigoplus_{i \le -2} \g(j)$, and $\chi(y) = 0$ so
that $\chi(g \cdot x) = \chi(x)$.  It
follows that $g \cdot x \in \m$ and that $g \cdot (x - \chi(x)) = g \cdot x - \chi(g \cdot x) \in \m_\chi$.
\end{proof}

The above lemma allows us to define an adjoint action of $\n$ on $Q$ by
$$
x \cdot (u + I) = [x,u]+I,
$$
for $x \in \n$ and $u \in U(\g)$.  Also we can define
an adjoint action of $N$ on $Q$ by
$$
g \cdot (u + I) = (g \cdot u) + I,
$$
for $g \in N$ and $u \in U(\g)$.  We note that $Q$ is a locally finite $N$-module for this adjoint
action, and the differential of this $N$-module structure coincides with
the $\n$-module structure on $Q$ given by the adjoint action.

\begin{Lemma} \label{L:invariantsareanalgebra}
The invariant space
$$
Q^\n := \{u + I \in Q \mid [x,u] \in I \text{ for all } x \in \n\}
$$
inherits an algebra structure from $U(\g)$ with the
$N$-invariants
$$
Q^N := \{u + I \in Q \mid g \cdot u + I = u + I \text{ for all } g \in N\}
$$
embedded as a subalgebra.
\end{Lemma}

\begin{proof}
We need to check that the multiplication on $Q^\n$ given by $(u + I)(v + I) = uv + I$ is well defined, and this
is done in the introduction of \cite{GG}.

Since the $\n$-module structure on $Q$ is given by the differential of the $N$-module structure,
we see that $N$ acts on $Q^\n$
and that $Q^N \sub Q^\n$.   Moreover, the action of $N$ on $Q^\n$ is by algebra automorphisms, so that $Q^N$ is a
subalgebra of $Q^\n$.
\end{proof}

This leads to the definition of the modular finite $W$-algebra.

\begin{Definition} \label{D:walgebra}
The
{\em finite $W$-algebra} associated to $e$ is the algebra
$$
U(\g,e) := Q^N.
$$
The {\em extended finite $W$-algebra} is
$$
\hU(\g,e) := Q^\n.
$$
\end{Definition}

The reader will notice that the definition of $U(\g,e)$ appears to depend upon the choice of
good grading for $e$ and the choice of isotropic subspace $\l$. In Section~\ref{S:indep}
we show that the isomorphism type really only depends on the orbit of $e$. By contrast,
the isomorphism type of the algebra $\hU(\g,e)$ depends on the dimension of $\n$.
From Section~\ref{S:applicationsofPBW} onwards we just consider the case
where $\l$ is a Lagrangian subspace of $\g(-1)$, and in this case Theorem~\ref{T:extendedpcentre}
shows that $\hU(\g,e)$ is then independent of this choice, up to isomorphism.

We note that the action of $N$ on $Q$ restricts to an action of $N$ on $\hat U(\g,e)$.
Then we have that $U(\g,e)$ is equal to the $N$-invariants in $\hat U(\g,e)$.  We note that
this approach has been considered by Premet in \cite[Remark~2.1]{PrGR} for $p$ sufficiently large,
and in that case the role of $U(\g,e)$ is played by the reduction modulo $p$ of the characteristic
zero $W$-algebra.

\subsection{\texorpdfstring{Realising $U(\g,e)$ as endomorphisms of $Q$}
{Realising U(g,e) as endomorphisms of Q}}
A common definition of the finite $W$-algebra over $\C$ is the endomorphism
algebra of the generalised Gelfand--Graev module, and here we record a similar description of $U(\g,e)$.
We denote the algebra of endomorphisms of $Q$ by
$\End_{U(\g)}(Q)$ and denote its opposite algebra by $\End_{U(\g)}(Q)^\op$.
The adjoint action of $N$ on $Q$ induces an action on  $\End_{U(\g)}(Q)^\op$ by
$(g \cdot f)(u+I) = g \cdot (f(g^{-1} \cdot u +I))$ for $g \in N$, $f \in \End_{U(\g)}(Q)$, $u \in U(\g)$,
and the invariant subalgebra is denoted $\End_{U(\g)}^N(Q)^\op$.

\begin{Lemma}
There is a natural algebra isomorphism
$$
\End_{U(\g)}^N(Q)^\op \isoto U(\g,e).
$$
\end{Lemma}

\begin{proof}
By Frobenius reciprocity the map
$\End_{U(\g)}(Q)^\op \to Q^\m \sub Q^\n$ given by
$f \mapsto f(1+I)$ is an isomorphism.  Moreover,
this isomorphism intertwines the actions of
$N$ on $\End_{U(\g)}(Q)$ and $Q^\m$.  Hence, we obtain the required
isomorphism.
\end{proof}

We also note that if $\l$ is chosen to be a Lagrangian subspace of $\g(-1)$, then
the proof above shows that $\hU(\g,e)$ is isomorphic to $\End_{U(\g)}(Q)^\op$.

\section{\texorpdfstring{The PBW theorem for $U(\g,e)$ and $\hU(\g,e)$}{The PBW theorem for U(g,e) and hatU(g,e)}}

In this section we prove the PBW theorem for $U(\g,e)$,
which says that it is a filtered deformation of the coordinate ring of a good transverse slice
to the orbit of $e$.  We also prove a PBW theorem for $\hU(\g,e)$.
We continue to use the notation from the previous section; in particular, we have
fixed a good grading $\Gamma_\gamma : \g = \bigoplus_{j \in \frac{2}{d}\Z} \g(j)$ and a compatible isotropic subspace $\l \sub \g(-1)$.

\subsection{Good transverse slices} \label{ss:goodtransverse}

Over fields of characteristic zero the Jacobson--Morozov theorem allows one to construct
transverse slices to all nilpotent orbits in Lie algebras of reductive algebraic groups,
as done by Slodowy in \cite[\S7.4]{Sl}.
Over fields of positive characteristic, such a choice of slice always exists
but it may fail to be transversal.
The correct replacement for the Slodowy slice is a good transverse slice, as was introduced
by Spaltenstien in \cite{Sp}, and we now recall.

The space $[\g, e]$ is stable under the action of $T_{[e]}$, so
decomposes as in \eqref{e:goodrestricted} and we may
choose a $T_{[e]}$-stable complement $\v$ to $[\g,e]$ in $\g$. We call $e + \v$ a {\em good transverse slice}
to $G \cdot e$.

Recall that the cocharacter $\gamma \in X_*(T^{[e]})$ defines a good grading for $e$ and that $\gamma(t)e = t^d e$,
where $d = \deg_\gamma(e) \in \Z_{>0}$.
We define the cocharacter $\mu : \kk^\times \to \GL(\g)$ by
\begin{equation} \label{e:muaction}
\mu(t)(x) = t^d(\gamma(t^{-1}) \cdot x),
\end{equation}
so, explicitly, $\mu(t)$ acts on $\g(j)$ by $t^{d-\frac{jd}{2}}$. By the definition of a good grading we have
$\bigoplus_{j \geq 1} \g(j) \sub [\g,e]$ and so $\v \sub \bigoplus_{i < 1} \g(i)$.
It follows that $\mu(\kk^\times)$ preserves $e + \v$ and acts with strictly positive eigenvalues on $\v$.
Therefore, $\mu$ defines a contracting $\kk^\times$-action with unique fixed point $e$.  Also we observe that
$\v \sub \m^\perp$, because $\m \sub \bigoplus_{i \le -1} \g(i)$ so that
$\bigoplus_{i < 1} \g(i) \sub \m^\perp$; here we recall that $\m^\perp$ is the orthogonal space to $\m$ with
respect to the form $(\cdot\,,\cdot)$.

Since $\g^e$ is orthogonal to $[\g,e]$ with respect to $(\cdot\,,\cdot)$, we see that the map
$\g^e  \to \v^*$ given by $x \mapsto (x,\cdot)$ is an isomorphism, or
in other words that the form $(\cdot, \cdot)$ restricts to a non-degenerate pairing
\begin{equation} \label{e:vdual}
\g^e \times \v \to \kk.
\end{equation}
This can be observed by noting that $[\g,e]$ is orthogonal to $\g^e$, and
then using the fact that $\v$ is a complement to $[\g,e]$ in $\g$.

We have the following important lemma about the good transverse slice $e + \v$.
For the statement we observe that $e + \m^\perp$ is stable under the adjoint action of $N$.
This can be confirmed by first noticing that $\m_\chi$ can be interpreted as a set of functions on $\g$ via $(\cdot\,,\cdot)$
whose zero-set is precisely $e+\m^\perp$, and then recalling that
$\m_\chi$ is $N$-stable thanks to Lemma~\ref{L:Npreservesm}.
The next lemma can be proved using the approach in \cite[Lemma~2.1]{GG}, which adapts to our setting,
as was observed in \cite[Proposition~6.9]{Ta};
see also \cite[Lemma~3.2]{PrCQ} for $p \gg 0$.  The contracting
$\kk^\times$-action from \eqref{e:muaction} is key to the argument.

\begin{Lemma}\label{L:isomorphismlemma}
The adjoint action induces an isomorphism of varieties
$$
N \times (e + \v) \isoto e + \m^\perp.
$$
\end{Lemma}

\subsection{The Kazhdan filtration} \label{ss:kazhdan}

We proceed as in \cite[Section~4]{GG} to introduce the Kazhdan filtration on $U(\g,e)$, which is derived from our good
grading. The {\em Kazhdan filtration} $(F_j U(\g))$
of $U(\g)$ is the $\frac{2}{d}\Z$-filtration
where $F_j U(\g)$ is spanned by monomials $x_1 \cdots x_k$ with $x_i \in \g(n_i)$ and
\begin{eqnarray}
\label{e:kazfirst}
\sum_{i=1}^k n_i + 2k \leq j.
\end{eqnarray}
This makes $U(\g)$ into a filtered algebra, and the associated graded algebra is denoted $\gr U(\g)$.  We have
that $\gr U(\g)$ is isomorphic to the symmetric algebra $S(\g)$
with $\g(j)$ lying in graded degree $j + 2$.  We write this grading as $S(\g) = \bigoplus_{j \in \frac{2}{d}\Z} S(\g)[j]$.
We note that since we are working with a filtration by $\frac{2}{d}\Z$, we require minor modifications
to the definitions of filtered algebras and associated graded algebras, but as these modifications
are clear we omit the details.
The Kazhdan filtration descends to all subquotients of $U(\g)$ in the usual manner, so that $I$,
$Q$, $U(\g,e)$ and $\hU(\g,e)$ inherit the Kazhdan filtration.

Using $(\cdot\,,\cdot)$ we may identify $S(\g) = \kk[\g]$.
Under this identification the grading on $\kk[\g]$ is given by
\begin{equation} \label{e:kggrading}
\kk[\g][j] = \{f \in \kk[\g] \mid f(\mu(t) \cdot x) = t^{\frac{dj}{2}}f(x) \text{ for all } t \in \kk^\times, x \in \g\}.
\end{equation}
We have that $\gr Q \cong \gr U(\g)/ \gr(I)$,
and $\gr I$ is the ideal of $\gr U(\g) = S(\g)$ generated by
$\m_\chi = \{x - \chi(x) \mid x \in \m\}$.
Thus under the identification $S(\g) = \kk[\g]$, we have
$\gr Q = \kk[e+\m^\perp]$, because $\gr I$ identifies with the ideal of functions
vanishing on $e + \m^\perp$.
Further, the action of $N$ on $e + \m^\perp$ induces an action of $N$ on $\kk[e+\m^\perp]$
by automorphisms.

There is a {\em Kazhdan grading} on $\kk[N]$ determined from the conjugation action of
$\gamma(t)$ on $N$, for $t \in \kk^\times$; so for $j \in \Z$ the $j$th graded piece
is given by
\begin{equation} \label{e:kNgrading}
\kk[N][j] := \{f \in \kk[N] \mid f(\gamma(t)^{-1}n\gamma(t)) = t^{\frac{dj}{2}}f(n) \text{ for all } t \in \kk^\times, n \in N\}.
\end{equation}
Using this grading
$Q$ is a Kazhdan filtered $N$-module, in the sense that the
comodule map $Q \to Q \otimes \kk[N]$ is filtered.  By taking the associated
graded map $\gr Q \to \gr Q \otimes \kk[N]$, we see that $\gr Q$ obtains
the structure of a Kazhdan graded $N$-module.  Moreover, this agrees with the
action of $N$ on $\kk[e+\m^\perp]$ through the identification $\gr Q = \kk[e+\m^\perp]$.

By Lemma~\ref{L:isomorphismlemma}, we have isomorphisms
\begin{equation} \label{e:emperp}
\gr Q = \kk[e+\m^\perp] \cong \kk[N \times (e+\v)] \cong \kk[N] \otimes \kk[e+\v].
\end{equation}
Here the gradings on $\kk[e+\m^\perp]$ and $\kk[e+\v]$ are determined from \eqref{e:kggrading}
and the grading on $\kk[N]$ is given by \eqref{e:kNgrading}.
Moreover, through this isomorphism the action of $N$ on $\kk[N] \otimes \kk[e+\v]$ is
given by $x \cdot (f \otimes g) = (x \cdot f) \otimes g$, where
\begin{equation} \label{e:Naction}
(x \cdot f)(n) = f(x^{-1}n)
\end{equation}
for $x,n \in N$, $f \in \kk[N]$ and $g \in \kk[e+\v]$.

\subsection{\texorpdfstring{PBW theorem for $U(\g,e)$}{PBW theorem for U(g,e)}}
In order to prove the PBW theorem for $U(\g,e)$, we need to
look at the cohomology of $\kk[N]$ as a module for $N$
with the action of $N$ on $\kk[N]$ by left translation given in \eqref{e:Naction}.
This cohomology $H^i(N,\kk[N])$ is given in \cite[Lemma~4.7]{JaAG} by
\begin{equation} \label{e:NcohomkN}
H^i(N, \kk[N]) = \left\{\begin{array}{cc} \kk & \text{ if } i = 0 \\ 0 & \text{ if } i > 0.\end{array}\right.
\end{equation}
With this in hand, we are ready to prove the PBW theorem for $U(\g,e)$.  Our proof
follows the same lines as that in \cite[Theorem~4.1]{GG}, but with $N$-cohomology
in place of $\n$-cohomology.

\begin{Theorem} \label{T:PBWtheorem}
The natural map
$\gr(Q^N) \to (\gr Q)^N$
is an isomorphism. Thus we have an isomorphism $\gr U(\g,e) \cong \kk[e + \v]$.
\end{Theorem}

\begin{proof}
For the duration of this proof we rescale all of the gradings and filtrations defined in
\eqref{e:kazfirst}, \eqref{e:kggrading}, \eqref{e:kNgrading} by a factor of $d/2$ so that we
obtain $\Z$-filtrations and $\Z$-gradings (the reason for this is simply to align our notation
with the conventions in the theory of spectral sequences).
Hence we have $Q = \bigcup_{i \in \Z_{\ge 0}} F_i Q$ and $\kk[N] = \bigoplus_{i \in \Z_{\ge 0}} \kk[N][i]$
and the isomorphism \eqref{e:emperp} is now homogeneous
with respect to the $\Z$-gradings.

The complex for computing the cohomology of $Q$ as an $N$-module
given in \cite[\S4.14]{JaAG} has $r$th term
$$
C^r(N,Q) = Q \otimes \kk[N]^{\otimes r}
$$
and the differential $C^r(N,Q) \to C^{r+1}(N,Q)$ is a certain signed sum of
maps determined from the comodule map $Q \to Q \otimes \kk[N]$ and the comultiplication
$\kk[N] \to \kk[N] \otimes \kk[N]$. We filter this complex by $F_i C^r(N, Q) := (F_i Q) \otimes \kk[N]^{\otimes r}$
for $i \in \Z$. It follows from the definitions that the differentials respect this filtration and that
$\gr C^r(N, Q) \cong C^r(N, \gr Q)$ is the complex for calculating the $N$-cohomology of $\gr Q$.

There is a standard spectral sequence for computing cohomology of a filtered complex, for which the first page is
 \begin{eqnarray}
 \label{e:thesequence}
 E_1^{r,s} := H^{r+s}( F_r C(N, Q) / F_{r-1} C(N, Q)) \cong H^{r+s} (C(N, F_r Q / F_{r-1} Q))
 \end{eqnarray}
 and differentials $\partial_1 : E_1^{r,s} \rightarrow E_1^{r -1, s+ 2}$.
As explained in \S\ref{ss:kazhdan}, we have the identification $\gr Q = \kk[e+\m^\perp]$ and
the isomorphism $\kk[e+\m^\perp]  \cong \kk[N] \otimes \kk[e+\v]$ of graded $\kk[N]$-comodules.
Thus from \eqref{e:NcohomkN}
we deduce that
\begin{eqnarray}\label{e:thecohom}
H^i(N,\gr Q) \cong \left\{\begin{array}{cc} \kk[e + \v] & \text{ if } i = 0 \\ 0 & \text{ if } i > 0.\end{array}\right.
\end{eqnarray}
Combining \eqref{e:thesequence} and \eqref{e:thecohom} we have $E_1^{r, s} = 0$ unless $r + s = 0$. It follows that the differentials $\partial_1$ are all zero,
and the spectral sequence collapses on the first page, which immediately implies $\gr H^{r}(C(N, Q)) \cong H^r(\gr C(N, Q)))$ for all $r$.
In particular, the natural map $\gr H^0(N,Q) \to H^0(N,\gr Q)$ is an isomorphism. Since, $H^0(N,Q) = Q^N = U(\g,e)$ and $(\gr Q)^N = H^0(N,\gr Q) \cong \kk[e+\v]$,
we deduce the required isomorphisms.
\end{proof}

\subsection{\texorpdfstring{The PBW theorem for $\hU(\g,e)$}{The PBW theorem for hatU(g,e)}}
We move on to prove a PBW theorem for the extended finite $W$-algebra
$\hU(\g,e)$, which is stated in
Theorem~\ref{T:extendedPBWtheorem} below.  Here our methods
are more explicit, as we can show that $\kk[N]$ is free as a module
over $U_0(\n)$, and this is the key ingredient of our proof.

We consider the action of $\n$ on $\kk[N]$ given by taking the differential
of the $N$-module structure given by \eqref{e:Naction}; we denote this action
by $(x,f) \mapsto x \cdot f$ for $x \in \n$ and $f \in \kk[N]$.
We note that this differential
is defined as $\kk[N]$ is a locally finite as an $N$-module, and
gives an action of the restricted enveloping algebra $U_0(\n)$ on $\kk[N]$.
We use the setup from Remark~\ref{R:char0a}, so that as in \cite[\S9.6]{JaAG},
we can identify $U_0(\n)$ with the algebra of distributions of the first
Frobenius kernel $N_1$ of $N$.
Through this identification the structure of $\kk[N]$ as an $N_1$-module
is just given by the restriction of $\kk[N]$ as an $N$-module.

We enumerate $\Phi(\n) = \{\alpha_1,\dots,\alpha_l\}$.
Elements of $N$ can be written uniquely in the form
$n = \prod_{i=1}^l u_{\alpha_i}(t_i)$, where $t_i \in \kk$, and we define $T_i \in \kk[N]$ by
$T_i(n) = t_i$.  Then $\kk[N] = \kk[T_1,\dots,T_l]$.  For $\ba = (a_1,\dots,a_l) \in \Z_{\ge 0}^l$,
we define $T^\ba = \prod_{i=1}^l T_i^{a_i}$.

Let $J_1$ be the ideal of $\kk[N]$ generated by $\{T_1^p,\dots,T_l^p\}$, so that
$\kk[N_1] = \kk[N]/J_1$.
The action of $U_0(\n)$ on $\kk[N]$ factors to an action on $\kk[N_1]$ and this
agrees with the left regular action of $N_1$ on $\kk[N_1]$.  Thus by \cite[Lemma~4.7]{JaAG} we have that $H^0(N_1,\kk[N_1]) = \kk$.

We consider $U_0(\n) \cdot \bar T^\bv \sub \kk[N_1]$ where $\bv = (p-1,\dots,p-1)$ and $\bar T^\bv$ denotes the image of $T^\bv$ in $\kk[N_1]$.
We note that the good grading on $\g$ gives a non-positive grading of $U_0(\n)$, which we write as
$U_0(\n) = \bigoplus_{j \in \frac{2}{d} \Z_{\le 0}} U_0(\n)(j)$, with $U_0(\n)(0) = \kk$.
Let $m \in \frac{2}{d} \Z_{\ge 0}$ be maximal such that
there exists $u \in U_0(\n)(-m)$ with $u \cdot \bar T^\bv \ne 0$.  Then we have
$u \cdot \bar T^\bv$ is annihilated by all elements of $\n$ and is thus an element of
$H^0(N_1,\kk[N_1]) = \kk$.  We deduce that up to scalar multiples, we must have that
$u = e^{\bv} := \prod_{i=1}^l e_{\alpha_i}^{p-1}$ and that $e^{\bv} \cdot \bar T^{\bv} = c$ is a nonzero
element $c \in \kk$.  A standard argument, based on weights for $T$,
shows that $U_0(\n)$ viewed as the left regular module for itself has 1-dimensional socle spanned by
$e^{\bv}$.  From this it follows that $\Ann_{U_0(\n)}(\bar T^\bv) = \{0\}$.
Thus, as $\dim \kk[N_1] = \dim U_0(\n)$, we have that $\kk[N_1] = U_0(\n) \cdot \bar T^\bv$ is free of rank 1 as a $U_0(\n)$-module.

For $\bb,\bc \in \Z_{\ge 0}^l$ and $u \in U_0(\n)$, we have that
$u \cdot T^{p\bb+\bc} = T^{p\bb}(u \cdot T^{\bc})$.  Therefore,
the cyclic module $U_0(\n) \cdot T^{p\bb+\bv} \sub \kk[N]$ generated by $T^{p\bb+\bv}$ is
free for any $\bb \in \Z_{\ge 0}^l$.




Let $\Z_{[0,p)}^l = \{\ba = (a_1,\dots,a_l) \in \Z^l \mid  0 \le a_i < p\}$
and define $\kk[N]_\mathrm{res}$ to be the span of all $T^\ba$ for $\ba \in \Z_{[0,p)}^l$.
Take $f \in \kk[N]^\n$ and write it as $f = \sum_{\ba \in \Z_{\ge 0}^l} T^{p\ba} f_\ba$, where
$f_\ba \in \kk[N]_\mathrm{res}$.
Given $x \in \n$ we have
$0 = x \cdot f = \sum_{\ba \in \Z_{\ge 0}^l} T^{p\ba}(x \cdot f_\ba)$.
Using the fact that $U_0(\n) \cdot \bar T^\bv \cong \kk[N_1]$ is a free $U_0(\n)$-module,
we can recursively deduce that $x \cdot f_\ba = 0$ for each $\ba \in \Z_{\ge 0}^l$,
and thus that $f_\ba \in \kk$.  Hence, we deduce that $\kk[N]^\n = \kk[N]^p$.

We define $e^\ba = \prod_{i=1}^s e_{\alpha_i}^{a_i}$ for $\ba \in \Z_{[0,p)}^l$ and consider
$$
\mathfrak B = \{e^\ba \cdot T^{p\bb+\bv} \mid \ba \in \Z_{[0,p)}^l, \bb \in \Z_{\ge 0}^l\}.
$$
Again using the fact that $U_0(\n) \cdot \bar T^\bv = \kk[N_1]$ is a free $U_0(\n)$-module,
we have that $\mathfrak B$ is a linearly independent set.
For any $j \in \Z$, the number of elements
of $\mathfrak B$ with degree $j$ with respect to grading on $\kk[N]$ given
by \eqref{e:kNgrading}
is equal to the number of monomials in $\{T^\bc \mid \bc \in \Z_{\ge 0}^s\}$
of degree $j$.  Therefore, we see that $\mathfrak B$ is in fact a basis of $\kk[N]$.

Putting this all together we deduce that
\begin{equation} \label{e:kNfree}
\kk[N] = \bigoplus_{\bb \in \Z_{\ge 0}^l} U_0(\n) \cdot T^{p\bb+\bv}
\end{equation}
so that $\kk[N]$ is free over $U_0(\n)$.
We have seen that the $\n$-invariants in $\kk[N]$ are
given by $\kk[N]^\n = \kk[N]^p$.  From this we see that we must have
\begin{equation} \label{e:e^vT}
e^\bv \cdot T^{p\bb+\bv} = c T^{p\bb}
\end{equation}
for any $\bb \in \Z_{\ge 0}^l$, where as before $c$ is some nonzero element of $\kk$.

Now that we have shown that $\kk[N]$ is free as a $U_0(\n)$-module, we are ready to prove the
PBW theorem for $\hU(\g,e)$.

\begin{Theorem}\label{T:extendedPBWtheorem}
The natural map $\gr(Q^\n) \to (\gr Q)^\n$ is an isomorphism.
Thus we have an isomorphism $\gr \hU(\g,e) \cong \kk[N]^p \otimes \kk[e + \v]$.
\end{Theorem}

\begin{proof}
We have $\gr Q \cong \kk[N] \otimes \kk[e+\v]$ from \eqref{e:emperp}.
Also thanks to \eqref{e:vdual}, we may identify $\kk[e+\v]$ with $S(\g^e)$,
and we make the identification $\gr Q = \kk[N] \otimes S(\g^e)$ throughout the proof.
We pick a basis $x_1,\dots,x_r$ of $\g^e$ such that $x_i \in \g(n_i)$ for $n_i \in \frac{2}{d}\Z_{\ge 0}$
and define $x^\bc = x_1^{c_1}\dots x_r^{c_r}$ for $\bc = (c_1,\dots,c_r) \in \Z_{\ge 0}^r$.
Then the Kazhdan degree of $x_i$ is $n_i+2$ and $\{x^\bc \mid \bc \in \Z_{\ge 0}^r\}$ is a basis of $S(\g^e)$.
Using \eqref{e:kNfree}
we see that $\kk[N] \otimes S(\g^e)$ is
a free $U_0(\n)$-module on
$\{T^{p\bb+\bv} \otimes x^\bc \mid \bb \in \Z_{\ge 0}^l, \bc \in \Z_{\ge 0}^r\}$.
By \eqref{e:e^vT}, we have that
$e^\bv \cdot (T^{p\bb+\bv} \otimes x^\bc) = c T^{p\bb} \otimes x^\bc$ for any $\bb \in \Z_{\ge 0}^l$, $\bc \in \Z_{\ge 0}^r$, where
$c$ is a nonzero element of $\kk$.
Thus we deduce that the $\n$-invariants in $\kk[N] \otimes S(\g^e)$ has basis
$\{T^{p\bb} \otimes x^\bc \mid \bb \in \Z_{\ge 0}^l, \bc \in \Z_{\ge 0}^r\}$, so that
$$
(\gr Q)^\n \cong (\kk[N] \otimes S(\g^e))^\n = \kk[N]^p \otimes S(\g^e).
$$

We choose a lift $u(\bb,\bc) \in Q$ of $T^{p\bb+\bv} \otimes x^\bc$,
for each $\bb \in \Z_{\ge 0}^s$, $\bc \in \Z_{\ge 0}^r$.
Now a standard filtration argument shows that $Q$ is a free $U_0(\n)$-module on
$\{u(\bb,\bc) \mid \bb \in \Z_{\ge 0}^l, \bc \in \Z_{\ge 0}^r\}$.
From this we deduce that $Q^\n$ has basis  $\{e^\bv \cdot u(\bb,\bc) \mid \bb \in \Z_{\ge 0}^l, \bc \in \Z_{\ge 0}^r\}$.

The natural map $\gr(Q^\n) \to (\gr Q)^\n$ sends
$\gr_j (e^\bv \cdot u(\bb,\bc)) \in \gr(Q^\n)$, where $j = \sum_{i=1}^l (pb_i-(p-1))m_i + \sum_{i=1}^r c_i(n_i+2)$,
to $e^\bv \cdot (T^{p\bb+\bv} \otimes x^\bc) = c T^{p\ba} \otimes x^\bc$, and thus
is an isomorphism.

Since $\hU(\g,e) = Q^\n$ and we have already seen that $(\gr Q)^\n \cong \kk[N]^p \otimes S(\g^e)$,
we have $\gr \hU(\g,e) \cong \kk[N]^p \otimes S(\g^e) = \kk[N]^p \otimes \kk[e + \v]$.
\end{proof}

\section{Independence results}\label{S:indep}

In this section we show that the isomorphism type of $U(\g,e)$ does not depend on the choice of isotropic space
$\l \sub \g(-1)$ or the choice of good grading.  This is achieved by adapting the methods from \cite{GG} and \cite{BruG}
to the modular setting.

\subsection{Independence of the isotropic space}
We fix a good cocharacter $\gamma$ for $e$ giving a good grading
$$
\Gamma_\gamma : \g = \bigoplus_{j \in \frac{2}{d}\Z} \g(j).
$$
Now consider two isotropic subspaces $\l,\l' \sub \g(-1)$ with $\l' \sub \l$.
As we have different isotropic subspaces in play we have to be more careful with notation:
we write $N_\l$, $Q_\l$ and $U(\g,e)_\l$ for these objects defined with the choice $\l$; and
$N_{\l'}$, $Q_{\l'}$ and $U(\g,e)_{\l'}$ for the corresponding objects defined with the choice $\l'$.

The projection $Q_{\l'} \onto Q_\l$ restricts to
a Kazhdan filtered algebra homomorphism
$$
U(\g,e)_{\l'} = Q_{\l'}^{N_{\l'}} \to Q_\l^{N_\l} = U(\g,e)_\l.
$$
By Theorem~\ref{T:PBWtheorem} the associated graded algebras are both isomorphic to $\kk[e + \v]$,
and from this we can conclude that the homomorphism above
is actually an isomorphism.

By considering the case $\l' = 0$, we see that $U(\g,e)_\l$ does not depend on the choice of
$\l$, so we obtain.

\begin{Proposition} \label{P:indepl}
For a fixed good grading $\Gamma_\gamma$,
the isomorphism type of the finite $W$-algebra $U(\g,e)$ does not depend
on the choice of isotropic space $\l$.
\end{Proposition}

\subsection{Independence of the good grading}
We proceed to deduce from Proposition~\ref{P:indepl} that the finite $W$-algebra
$U(\g,e)$ does not depend on the choice of good grading for $e$.  We use the
notation from Section~\ref{S:goodgrading}, and adapt the setup from \cite{BruG}.

Let $\gamma, \gamma' \in X_*(T_{[e]}) \setminus X_*(T_e)$ be good cocharacters for $e$.  Up to equivalence by $\sim$ as
defined in Section~\ref{S:goodgrading}, we may assume that
$\deg_\gamma(e) = 2c = \deg_{\gamma'}(e)$, where
$c \in \Z_{>0}$.
We write
$$
\Gamma_\gamma : \g = \bigoplus_{j \in \frac{1}{c}\Z} \g(j) \quad \text{and} \quad \Gamma_{\gamma'} : \g = \bigoplus_{j \in \frac{1}{c}\Z} \g'(j)
$$
for the corresponding good gradings of $\g$.

We say that $\gamma$ and $\gamma'$ are {\em adjacent} if
$$
\g = \bigoplus_{i^- \le j \le i^+} \g(i) \cap \g'(j).
$$
Here we are summing over $i,j \in \frac{1}{c}\Z$ satisfying $i^- \le j \le i^+$,
where $i^-$ denotes the largest integer strictly smaller than $i$ and $i^+$
denotes the smallest integer strictly larger than $i$.

The definition of adjacency is set up so that if $\gamma$ and $\gamma'$ are adjacent,
then there exist Lagrangian subspaces
$\l \sub \g(-1)$ and $\l' \sub \g'(-1)$ (compatible with $T$ and $\chi$) such that
$$
\l \oplus \bigoplus_{j < -1} \g(j) = \l' \oplus \bigoplus_{j < -1} \g'(j);
$$
this can be proved in exactly the same way as \cite[Lemma~26]{BruG}.
For this choice of $\l$ and $\l'$, the finite $W$-algebra $U(\g,e)$ defined from the good grading
$\Gamma_\gamma$ and the Lagrangian subspace $\l$ is actually {\em equal} to the finite $W$-algebra
$U(\g,e)'$ defined from the good grading $\Gamma_{\gamma'}$ and the Lagrangian subspace $\l'$.

For $\alpha \in \Phi^e$ and $k \in \Z$ we define the affine hyperplane
$$
H_{\alpha,k} := \{\theta \in E_e \mid \langle \alpha,\theta \rangle = k\} \sub E_e^\R := \R \otimes_\Z X_*(T_e).
$$
These hyperplanes slice $E_e^\R$ in to alcoves, i.e.\ the alcoves
are the connected components of the complement in $E_e^\R$ of the union of
all $H_{\alpha,k}$.  Further, we see that the good grading polytope $\cP_e^\Q$ from
\eqref{e:polytope} is sliced into finitely many alcoves, and that we can get between
any two points in $\cP_e^\Q$ by passing through finitely many walls.

According to the classification of good gradings in Section~\ref{S:goodgrading},
for each pair $\frac{1}{c} \otimes \delta, \frac{1}{c} \otimes \delta' \in \cP_e^\Q$
there are good cocharacters $c\lambda - \delta, c\lambda - \delta' \in X_*(T_{[e]})$ for $e$,
and corresponding $W$-algebras $U(\g,e)$ and $U(\g,e)'$.
As explained in \cite[Section~5]{BruG}, we have that $c\lambda - \delta$ and $c\lambda - \delta'$
are adjacent if and only if $\frac{1}{c} \otimes \delta$ and $\frac{1}{c} \otimes \delta'$ lie in the closure of the same alcove.
If this is the case, we can choose Lagrangian subspaces in such a way that $U(\g,e) = U(\g,e)'$, as explained above.

Since the defining equations of the walls (\ref{e:polytope}) have integral coefficients it follows that if the closures of two alcoves
of $E_e^\R$ have non-empty intersection then this intersection contains a point of $\cP_e^\Q$.
Hence, between every two points in $\cP_e^\Q$ there is a path in $E_e^\R$ which passes through finitely many
walls and the point in each wall that it passes through is an element of $\cP_e^\Q$.
Thus by successively applying Proposition~\ref{P:indepl}, we obtain the following independence theorem.

\begin{Theorem} \label{T:gradingindep}
Let $\gamma, \gamma' \in X_*(T^{[e]})$ be good cocharacters for $e$, and let $U(\g,e)$ and
$U(\g,e)'$ be the finite $W$-algebras defined from the corresponding good gradings.  Then we have
$U(\g,e) \cong U(\g,e)'$.
\end{Theorem}

\section*{Notation}

For ease of reference we recap and introduce the notation that we use in the sequel; this
is in addition the basic notation given in \S\ref{ss:notation}, \S\ref{ss:redgrps} and \S\ref{ss:nilporbits}.

Given Theorem~\ref{T:gradingindep}, we choose to restrict to an integral good grading
$\Gamma_\gamma : \g = \bigoplus_{j \in \Z} \g(j)$
from now, and thus we make
a restriction on the cocharacter $\gamma \in X_*(T_{[e]})$ below.
We also just work with the case where $\l$ is a Lagrangian subspace of $\g(-1)$, in which
case $\m$ and $\n$ coincide, and so now we write $M = N$ for the algebraic group $M \subseteq G$ with $\m = \Lie(M)$.
In case $p =2$, we actually assume that the good grading is even: this is justified
because $2$ is a good prime for $G$ only if all simple components of $G$ are of type $\mathrm A$, and so it follows
from \cite[Section~4]{EK} (see also \cite[Section~6]{BruG}) that there is always an even good grading; here
we also require the
fact that the parametrization of good gradings is the same
as in characteristic 0 as explained in Section~\ref{S:goodgrading}.

\begin{itemize}
\itemsep5pt

\item $e \in \g$ is a nilpotent element compatible with $T$ and $\chi := (e, \cdot) \in \g^*$.

\item
$d_\chi := \frac{1}{2} \dim G \cdot \chi = \frac{1}{2}\dim G \cdot e$.

\item $\lambda$ is the associated cocharacter for $e$ in $X_*(T_{[e]})$.

\item $\Phi^e \sub X^*(T_e)$ is the restricted root system and we pick a set
of positive roots $\Phi^e_+ \sub \Phi^e$.

\item We fix $\gamma \in X_*(T)$ to be a good cocharacter for $e$ with $\deg_\gamma(e) = 2$, and write
$\Gamma_\gamma : \g = \bigoplus_{j \in \Z} \g(j)$ for the (integral) good grading of $\g$ determined by $\gamma$. \\
If $p = 2$, then we pick $\gamma$ so that $\Gamma_\gamma$ is an even good grading for $e$.

\item Incorporating the restricted root decomposition we obtain
$$
\g = \bigoplus_{j \in \Z} \g_0(j) \oplus \bigoplus_{\substack{\alpha \in \Phi^e \\ j \in \Z}} \g_\alpha(j).
$$

\item Let $\l = \bigoplus_{\alpha \in \Phi^e_+} \g_\alpha(-1)$, which is a Lagrangian
subspace of $\g(-1)$. \\
Also let $\l'= \bigoplus_{\alpha \in \Phi^e_-} \g_\alpha(-1)$, so that $\g(-1) = \l \oplus \l'$.

\item $\m := \l \oplus \bigoplus_{j < -1} \g(j)$, which is a nilpotent subalgebra of $\g$.

\item  $\bfp := \l' \oplus \bigoplus_{j \geq 0} \g(j)$, which is just a subspace of $\g$.

\item Recall that $\v$ is a $T_{[e]}$-stable complement of $[\g,e]$ in $\g$ and
that $\v$ is dual to $\g^e$ via $(\cdot\,,\cdot)$, see \eqref{e:vdual}.

\item
We fix a basis $x_1,...,x_r,x_{r+1},...,x_m$ for $\bfp$ such that $x_1,...,x_r$ spans $\g^e$.  \\
We choose this so that $x_i \in \g_{\alpha_i}(n_i)$, where $n_i \in \Z_{\ge -1}$ and
$\alpha_i \in \Phi^e \cup \{0\}$. \\
We may and do choose $x_{r+1},...,x_m$ to be orthogonal to $\v$, and write $\bfa$ for the vector
space spanned by $x_{r+1},...,x_m$.

\item $\m_\chi := \{x - \chi(x) \mid x \in \m\}$, $I := U(\g)\m_\chi$ and $Q := U(\g)/I$.

\item $M$ is the unipotent subgroup of $G$ with Lie algebra $\m$.

\item $U(\g,e) := Q^M = \{u + I \in Q \mid g \cdot u +I = u+I \text{ for all } g \in M\}$.

\item The Kazhdan filtration of $U(\g)$ is defined by setting $F_j U(\g)$, for $j \in \Z$ to be spanned by monomials
$y_1 \cdots y_k$ with $y_i \in \g(m_i)$ and
$$
\sum_{i=1}^k (m_i + 2) \leq j.
$$
This induces non-negative filtrations $(F_j Q)_{j \in \Z_{\ge 0}}$ and $(F_j U(\g,e))_{j \in \Z_{\ge 0}}$
on $Q$ and $U(\g,e)$ respectively.

\item For $\ba = (a_1,...,a_m) \in \Z_{\geq 0}^m$  we write
$$
x^\ba := x_1^{a_1} \cdots x_m^{a_m},
$$
which we view both as an element of $U(\g)$ and of $S(\g)$.  \\
Then $\{x^\ba + I \mid \ba \in \Z_{\geq 0}^m\}$ is a basis of $Q$, and
$\{x^\ba  \mid \ba \in  \Z_{\geq 0}^m\}$ is a basis of $S(\bfp)$. \\
Define $|\ba| = \sum_{i=1}^m a_i$ and $|\ba|_e = \sum_{i=1}^m (n_i+2)a_i$ to be the total degree
and the Kazhdan degree of $x^\ba$.
\end{itemize}

\section{The PBW generators} \label{S:applicationsofPBW}

We interpret Theorem~\ref{T:PBWtheorem} (the PBW theorem for $U(\g,e)$)  more explicitly
to give some good choices of PBW generators.  We roughly
follow the approach in \cite[\S3.2]{BGK} and explain
how to adapt the methods from there; which in turn is
based on work of Premet in \cite{PrST} and \cite{PrJI}.
We also describe a PBW basis of $\hU(\g,e)$ and use this
to clarify the relationship between $U(\g,e)$ and $\hU(\g,e)$.

Recall that the associated graded algebra of $Q$ is $\gr Q = \gr S(\g) /\gr(I)$, and note that
by the PBW theorem we have $S(\g) = S(\bfp) \oplus \gr I$.  We let $\pr : S(\g) \to S(\bfp)$
be the projection along this direct sum decomposition.  This restricts to an isomorphism
between $\gr Q$ and $S(\bfp)$.  As explained in \S\ref{ss:gelfandwhittaker} the adjoint action of $M$ on $Q$
descends to an adjoint action on $\gr Q$ and in turn this gives a twisted action of $M$ on
$S(\bfp)$ defined by
$$
\tw(g) \cdot f = \pr(g \cdot f),
$$
for $g \in M$ and $f \in S(\bfp)$, where $g \cdot f$ denotes the usual adjoint action of $g$ on $f$ in $S(\g)$.
We write $S(\bfp)^{\tw(M)}$ for the invariants with respect to this action.  Thanks to
Theorem~\ref{T:PBWtheorem} the invariant algebra $S(\bfp)^{\tw(M)}$ is isomorphic to $\gr U(\g,e)$.

Recall that $x_1,...,x_r, x_{r+1},...,x_m$ is a basis of $\bfp = \l' \oplus \bigoplus_{j \geq 0} \g(j)$.
We have that $x_{r+1},...,x_m$ spans a complement to $\g^e$ inside
$\bfp$, which we denote by $\bfa$; further, $\bfa$ is orthogonal to $\v$.
So we have a direct sum decomposition $\bfp = \bfa \oplus \g^e$ and also a decomposition
$S(\bfp) = S(\g^e) \oplus \bfa S(\bfp)$
which induces a projection
\begin{equation} \label{e:zeta}
\zeta : S(\bfp) \onto S(\g^e).
\end{equation}
Recall that the form $(\cdot\,,\cdot)$ on $\g$ induces an isomorphism $\g \cong \g^*$. In turn, this induces a homomorphism
\begin{equation}\label{e:identifyalgs}
S(\bfp) \hookrightarrow S(\g) = \kk[\g^*] \cong \kk[\g] \twoheadrightarrow \kk[e + \m^\perp].
\end{equation}
We have $\m^\perp = [e,\l] \oplus \bigoplus_{i< 1} \g(i)$ and so it follows that the above composition is
an isomorphism.  In what follows we usually identify $S(\bfp)$ with $\kk[e + \m^\perp]$. The twisted adjoint action of $N$
on $S(\bfp)$ is defined
precisely so that (\ref{e:identifyalgs}) is an isomorphism of $M$-modules.
Then using the fact that $(\cdot\,,\cdot )$ restricts to a non-degenerate pairing $\g^e \times \v \rightarrow \kk$,
we may identify $S(\g^e)$ with $\kk[e + \v]$.

By construction the kernel of the restriction
$\kk[e+ \m^\perp] \twoheadrightarrow \kk[e+ \v]$ identifies with the ideal of $S(\bfp)$ generated by $\bfa$ and
it follows that the projection $\zeta$ from \eqref{e:zeta} identifies with this restriction map.
Thanks to Lemma~\ref{L:isomorphismlemma}, the map $\zeta: \kk[e+\m^\perp] \to \kk[e+\v]$
gives an isomorphism $\kk[e+\m^\perp]^N \isoto \kk[e+\v]$.  Now we can deduce the following reformulation of the
PBW theorem, which is an analogue of \cite[Lemma~3.1]{BGK}.

\begin{Lemma} \label{L:symmetricequivariance}
We have that $\gr U(\g,e) \cong S(\bfp)^{\tw(M)}$.  Moreover, the restriction of $\zeta$
is an isomorphism of graded algebras
$$
\zeta : S(\bfp)^{\tw(M)} \to S(\g^e).
$$
\end{Lemma}

The preceding observations can be presented in the following commutative diagram, where
all the maps are explained above:
\begin{center}
\begin{tikzpicture}[node distance=2.3cm, auto]
\pgfmathsetmacro{\shift}{0.3ex}
\node (A) {$S(\bfp)^{\tw(M)}$};

\node (H) [right of= A] {$ $};
\node (B) [right of= H] {$S(\bfp)$};
\node (C) [below of= B]{$S(\g^e)$};

\node (G) [right of= B] {$ $};

\node (D) [right of= G]{$\kk[e + \m^\perp]$};
\node (E) [below of= D]{$\kk[e+\v]$};
\node (I) [right of= D]{$ $};
\node (F) [right of= I]{$\kk[e+\m^\perp]^M$};

\draw[right hook->] (A) --(B) node[above,midway] {$ $};
\draw[->] (A) --(C) node[right,midway] {$\sim$};
\draw[->>] (B) --(C) node[above,midway] {$ $};
\draw[left hook->] (F) --(D) node[above,midway] {$ $};
\draw[->] (F) --(E) node[left,midway] {$\sim$};
\draw[->>] (D) --(E) node[above,midway] {$ $};

\draw[->] (B) --(D) node[above,midway] {$\sim$};
\draw[->] (C) --(E) node[above,midway] {$\sim$};

\end{tikzpicture}
\end{center}

We can also rephrase the PBW theorem for $\hU(\g,e)$ in similar terms.
The twisted action of $\m$ on $S(\bfp)$ is defined by
$$
\tw(x) \cdot f = \pr(x \cdot f),
$$
for $x \in \m$ and $f \in S(\bfp)$, where $x \cdot f$ is the usual adjoint action of $x$ on $f$ in $S(\g)$.

\begin{Lemma} \label{L:hcWdegrees}
We have that $\gr \hU(\g,e) \cong S(\bfp)^{\tw(\m)}$.  Moreover, as
a Kazhdan graded algebra it is a polynomial algebra on $\dim \bfp$ generators of degrees
$$
n_1,\dots,n_r,pn_{r+1},\dots,pn_m
$$
\end{Lemma}

\begin{proof}
The first statement follows from
Theorem~\ref{T:extendedPBWtheorem}.
Further, we deduce that $\gr \hU(\g,e) \cong \kk[M]^p \otimes \kk[e+\v]$ as graded algebras.
From the identification $\kk[e+\v] = S(\g^e)$, we see that the degrees of the generators of $\kk[e+\v]$ are
$n_1,\dots,n_r$.  Thus through the identification $\kk[M] \otimes \kk[e+\v] = S(\bfp) = S(\g^e) \otimes S(\bfa)$,
we get that the degrees of the generators of $\kk[M]$ are $n_{r+1},\dots,n_m$.  Now the result
follows.
\end{proof}

From now on we identify $\gr U(\g,e) = S(\bfp)^{\tw(M)}$, so we view $\gr_j u$ as an element
of $S(\bfp)^{\tw(M)}$ for $u \in F_j U(\g,e)$.  By Lemma~\ref{L:symmetricequivariance} there exist (non-unique) elements
$\Theta(x_i) \in F_{n_i + 2} U(\g,e)$
satisfying
\begin{equation} \label{e:choosingTheta}
\zeta(\gr_{n_i+2} \Theta(x_i)) = x_i \in S(\g^e).
\end{equation}
Moreover, we can choose $\Theta(x_i)$ to be a $T_e$-weight vector with weight $\alpha_i \in \Phi^e$, as
the isomorphism $\zeta$ is $T_e$-equivariant; in particular, $\Theta(x_i)$ is an eigenvector for $\sigma := \gamma(-1)$.
We can extend $\Theta$ linearly to a get map $\Theta : \g^e \to U(\g,e)$, which is $T_e$-equivariant.
Moreover, $\Theta$ satisfies the properties of the following theorem.

\begin{Theorem}\label{T:PBWbasisthm}
$ $
\begin{enumerate}
\item[(i)] The set $\{\Theta(x_i) \mid i = 1,...,r\}$ generates $U(\g,e)$ and the PBW
monomials $$\{ \Theta^\bb := \Theta(x_1)^{b_1} \cdots \Theta(x_r)^{b_r} \mid \bb \in \Z_{\ge 0}^r\}$$
form a basis of $U(\g,e)$.
\item[(ii)] We have
$$
\Theta(x_i)
=  x_i + \sum_{|\ba|_e \le n_i + 2} \lambda_{\ba,i} x^\ba
$$
where $\lambda_\ba \in \kk$ satisfy:
\begin{itemize}
\item $\lambda_\ba = 0$ whenever $|\ba|_e = n_i+2$ and $|\ba| = 1$; and
\item $\lambda_\ba = 0$ whenever $|\ba|_e$ has different parity to $n_i$.
\end{itemize}
\end{enumerate}
\end{Theorem}

\begin{proof}
Part (i) follows directly from Lemma~\ref{L:symmetricequivariance}, as does part (ii) except
from checking that the two conditions indeed imply that $\lambda_\ba = 0$.

For the first of these conditions we proceed as in the proof of \cite[Lemma~3.7]{BGK}.
Let $\hat x_i = \gr_{n_i+2} \Theta(x_i) \in S(\bfp)^{\tw(M)}$.
Since $\zeta(\hat x_i) = x_i$, we certainly have $\hat x_i \equiv y \bmod \bigoplus_{n > 1} S^n(\bfp)$,
where $y \in \g(n_i)$ and $\zeta(y) = x_i$.  We proceed to prove that actually
$y$ centralizes $e$, which implies that $y = x_i$.
This can be done in the same way as in \cite[Lemma~3.7]{BGK}, which we recap here.

Supposing $y \not\in \g^e$, we have $0 \neq [y,e] \in \g(n_i+2)$ and we may
find $z \in \g(-n_i-2) \sub \m$ such that $\chi([z,y]) = ([z,y],e) = (z,[y,e]) \neq 0$.  Since
$\hat x_i \in S(\bfp)^{\tw(M)} \sub S(\bfp)^{\tw(\m)}$, we have that
$\tw(z) \cdot \hat x_i = 0$.  We may write $\hat x_i = y + u$, where
$u \in \bigoplus_{n > 1} S^n(\bfp)$.  Then we have
\begin{align*}
\tw(z) \cdot \hat x &= \pr([z,y]+[z,u]) \\
&= \chi([z,y]) + \pr([z,u]).
\end{align*}
Next we note that $u$ is a sum of monomials of the form $x_{j_1} \cdots x_{j_s}$, where
$s > 1$ and $j_1 < \dots < j_s$, and so that $n_{j_k} < n_i$ for each $i$.  We have that $[z,x_{j_1} \cdots x_{j_s}]$ is a sum of monomials
$x_{j_1} \cdots x_{j_{k-1}} [z,x_{j_k}] x_{j_{k+1}} \cdots x_{j_s}$, and $[z,x_{j_k}] \in \bigoplus_{j < -2} \g(j)$, so that $\chi([z,x_{j_k}]) = 0$.
Hence, $\pr([z,u]) = 0$ so that $\tw(z) \cdot \hat x = \chi([z,y]) \neq 0$.  This contradiction means that we
must have $y \in \g^e$ as required.

Now we show that $\lambda_\ba = 0$ whenever $|\ba|_e$ has different parity to $n_i$. When $p \neq 2$
we just note that the eigenvalue of $\sigma = \gamma(-1)$ on $x^\ba$ is $(-1)^{|\ba|_e}$ and use the fact that $\Theta(x_i)$ is an eigenvector for $\sigma$.
For the case $p =2$
the automorphism $\sigma$ is trivial, however we have chosen an even good grading, so
$|\ba|_e$ is even for all monomials $x^\ba$ appearing.
\end{proof}

We include a consequence of these PBW theorem here.

\begin{Corollary} \label{C:Qfree}
As a right $U(\g,e)$-module $Q$ is free on the set of $x^\ba + I$, for  $\ba \in \Z_{\ge 0}^m$ with $a_1 = \dots = a_r = 0$.
\end{Corollary}

\begin{proof}
By a standard filtration argument it is enough to show that $\gr Q \cong S(\bfp)$ is free over
$\gr U(\g,e)$ with basis $\{x^\ba \mid \ba \in \Z_{\ge 0}^m, a_1 = \dots = a_r = 0\}$.
To prove this we let $\theta_i = \gr_{n_i+2} \Theta$, which we view as an element of $S(\bfp)$, and
define $\theta^\bb = \theta^{b_1}\cdots\theta^{b_r}$ for $\bb \in \Z_{\ge 0}^r$.  Now consider
$$
\mathfrak B = \{\theta^\bb x^\ba \mid \bb \in \Z_{\ge 0}^r, \ba \in \Z_{\ge 0}^m, a_1 = \dots = a_r = 0\}.
$$
By considering terms of lowest total degree, we see that $\mathfrak B$ is linearly independent.  To finish the
proof it suffices to observe that for each $n \in \Z_{\ge 0}$, the number of elements of $\mathfrak B$ of Kazhdan degree at
most $n$ is equal to the dimension of $F_n Q$.
\end{proof}

We note that the characteristic zero version of the corollary above
is proved as part (3) of the theorem in \cite{Sk}, but the proof there
is not directly applicable.  Another proof is given in \cite[Lemma~3.6]{Go},
and can be adapted to the present situation, but we included the details
above for convenience of the reader.

\begin{Remark} \label{R:gensandrels}
For the purposes of this paper, we only require the formula for the
PBW generators given in Theorem~\ref{T:PBWbasisthm}.  We note that much more can be
proved about the commutators between these generators.
More specifically, the
PBW generators $\Theta(x_1),...,\Theta(x_r)$ satisfy commutators relations
\begin{equation}\label{therelation}
[\Theta(x_i), \Theta(x_j)] = \Theta([x_i, x_j]) + F_{i,j}(\Theta(x_1),...,\Theta(x_r)) \bmod F_{n_i + n_j} U(\g,e)
\end{equation}
where $F_{i,j}$ is a polynomial with zero constant term and linear term.  This can be proved
using the similar arguments to those in the proof of \cite[Theorem~4.6]{PrST}.


In fact it is possible to prove an analogue of \cite[Theorem~3.8]{BGK} showing that
$U(\g,e)$ is a filtered deformation of $U(\g^e)$ with respect to another filtration
called the {\em loop filtration}.  This requires a more detailed look at formulae
for $\Theta(x_i)$, when $x_i \in \g^e(0)$ as in \cite[Lemma~2.3]{PrJI}, and is not required in this paper.
We do note though, that it is possible to deduce the above formulae for the commutators
from this result.
\end{Remark}

We would like a similar description of generators for $\hU(\g,e)$, and
the next lemma is helpful for this.  We note that this result
is essentially implicit in \cite[Theorem~2.1]{PrCQ}.

\begin{Lemma}\label{L:pcentreannQ}
$\Ann_{Z_p(\g)} Q$ is the ideal of $Z_p(\g)$ generated by $y^p - y^{[p]} - \chi(y)^p$ for $y \in \m$.
\end{Lemma}

\begin{proof}
Let $y_1,\dots,y_s$ be a basis of $\m$.  Then $y_i(1 + I) = \chi(y_i) + I$ for each $i$ and so
$(y_i^p - y_i^{[p]} - \chi(y_i)^p)(1 + I) =  \chi(y_i^{[p]}) + I  = 0 + I$.  It follows that
$y_i^p - y_i^{[p]} - \chi(y_i)^p \in \Ann_{Z_p(\g)} Q$.

By the PBW theorem the vectors
$(x_1^p - x_1^{[p]})^{a_1}\cdots (x_m^p - x_m^{[p]})^{a_m} + I \in Q$ with $a_i \in \Z_{\ge 0}$
are linearly independent.
We have that the set of all
$$
(y_1^p - y_1^{[p]}- \chi(y_1)^p)^{b_1}\cdots (y_s^p - y_s^{[p]} - \chi(y_s)^p)^{b_s}(x_1^p - x_1^{[p]})^{a_1}\cdots (x_m^p - x_m^{[p]})^{a_m}
$$
for $a_i,b_j \in \Z_{\ge 0}$
is a basis of $Z_p(\g)$, and from this we can deduce the result.
\end{proof}

In order to give a PBW basis of $\hU(\g,e)$
we define $\hTheta$ to be the $T^\chi$-equivariant linear map
$$
\hTheta : \bfa^{(1)} \to \hU(\g,e)
$$
by setting
$\hTheta(x_i) := x_i^p - x_i^{[p]} + I \in \hU(\g,e)$ for $i = r+1,...,m$; we recall that $\bfa^{(1)}$ is the
Frobenius twist of $\bfa$.
We use the notation
$$
\Theta^\ba \hTheta^\bb := \Theta(x_1)^{a_1} \cdots \Theta(x_r)^{a_r} \hTheta(x_{r+1})^{b_{r+1}}  \cdots  \hTheta(x_m)^{b_m},
$$
for $\ba \in \Z_{\ge 0}^r$, $\bb =(b_{r+1},\dots,b_m) \in \Z_{\ge 0}^{m-r}$, and
define $Z_p(\bfa)$ to be the subalgebra of $\hU(\g,e)$ generated by $\{\hTheta(x_j) \mid j = r+1,...,m \}$.
It follows from Lemma~\ref{L:pcentreannQ} that $Z_p(\bfa)$ is a polynomial algebra with generators of degrees $pn_{r+1},...,p n_{m}$.
We are now ready to give the PBW basis of $\hU(\g,e)$.

\begin{Lemma}
The set
$$
\{\Theta(x_i), \hTheta(x_j) \mid i = 1,...,r, j = r+1,...,m \}
$$
generates $\hU(\g,e)$ and the set of PBW
monomials
$$
\{ \Theta^\ba \hTheta^\bb  \mid 0 \leq \ba \in \Z_{\ge 0}^r, \bb \in \Z_{\ge 0}^{m-r} \}
$$
is a basis of $\hU(\g,e)$.
\end{Lemma}

\begin{proof}
Let $A$ be the subalgebra of $\hU(\g,e)$ generated by $Z_p(\bfa)$ and $U(\g,e)$.
Using Theorem~\ref{T:PBWbasisthm}(ii) we see that the set of PBW monomials above is linearly independent,
so that $A \cong Z_p(\bfa) \otimes U(\g,e)$.
Now using Lemma~\ref{L:hcWdegrees} we see that the generators of
$A$ have the same Kazhdan degrees as the generators of $\hU(\g,e)$, so we must have equality.
\end{proof}

We arrive at a comparison between $U(\g,e)$ and $\hU(\g,e)$.

\begin{Corollary}\label{C:UvsUhat}
$\hU(\g,e) \cong U(\g,e) \otimes Z_p(\bfa)$ as algebras and the following composition is an isomorphism
$$
U(\g,e) \into \hU(\g,e) \onto \hU(\g,e)/ \xi(\bfa^{(1)}) \hU(\g,e),
$$
where $\xi$ is defined in (\ref{e:ximap}).
\end{Corollary}

\section{\texorpdfstring{The $p$-centre of the modular finite $W$-algebras}{The p-centre of modular finite W-algebras}}
\label{S:pcentre}
As is the case for the enveloping algebra of a restricted Lie algebra, the representation theory
of $U(\g,e)$ is controlled by its $p$-centre.  We define this $p$-centre in Definition~\ref{D:pcentre}
and then proceed to prove a number of results about it.  After that we introduce reduced finite $W$-algebras
and prove that they coincide with those previously considered in work of Premet.

\subsection{\texorpdfstring{The definition and basic properties of the $p$-centre}
{The definition and basic properties of the p-centre}} \label{ss:pcentre}

Recall the isomorphism $\xi : S(\g)^{(1)} \rightarrow Z_p(\g)$ from (\ref{e:ximap}). By Lemma~\ref{L:pcentreannQ}, we have that
$I \cap Z_p(\g) = \xi(\m_\chi^{(1)})Z_p(\g)$, where $\xi(\m^{(1)}_\chi) = \{y^p - y^{[p]} - \chi(y)^p \mid y \in \m\}$, and we write $I_p$ for $I \cap Z_p(\g)$.
We can view $Z_p(\g)/I_p$ as a subalgebra
of $\hU(\g,e)$, as it is the image of the map $Z_p(\g) \to Q^\m = \hU(\g,e)$.
Since $\xi$ is $G$-equivariant and $\m_\chi$ is $M$-stable, by Lemma~\ref{L:Npreservesm},
we see that  the adjoint action of $M$ on $\hU(\g,e)$ preserves the subalgebra $Z_p(\g)/I_p$.

\begin{Definition} \label{D:pcentre}
The {\em $p$-centre $Z_p(\g,e)$ of $U(\g,e)$} is the invariant subalgebra
$$
Z_p(\g,e) := (Z_p(\g)/I_p)^M.
$$
\end{Definition}

In the sequel we will consider the following subvarieties of $\g^*$
\begin{align*}
\chi + \hat\m &:= \{\chi + (x, \cdot) \mid x \in \m^\perp\};\\
\chi + \check\v &:= \{\chi + (x, \cdot) \mid x \in \v\}.
\end{align*}
and their Frobenius twists which are subvarieties of $(\g^*)^{(1)}$. From the above remarks, the kernel of the composition
$\kk[\g^*]^{(1)} = S(\g)^{(1)} \isoto Z_p(\g) \twoheadrightarrow Z_p(\g)/I_p$
is generated by $\m_\chi^{(1)} \subseteq S(\g)^{(1)}$, which generates the defining ideal of $(\chi + \hat\m)^{(1)} \subseteq (\g^*)^{(1)}$.
Since $\g^*$ identifies with $(\g^*)^{(1)}$ as a $G$-set we may invoke Lemma~\ref{L:isomorphismlemma}
to see that the restriction map from $(\kk[\chi + \hat\m]^{(1)})^M$ to
$\kk[\chi + \check\v]^{(1)}$ is an isomorphism. The defining ideal of $(\chi + \check\v)^{(1)}$ in $\kk[\g^{\ast}]^{(1)}$
is generated by $\{x - \chi(x) \mid x \in (\m \oplus \bfa)^{(1)}\}$ and so we arrive at the next
lemma about the $p$-centres.  For the statement of this lemma, we define
$\bar I_p := \xi((\bfa \oplus \m_\chi)^{(1)})Z_p(\g)$ to be
the ideal of $Z_p(\g)$ generated by $\{x^p - x^{[p]} - \chi(x)^p \mid x\in  \m\oplus\bfa\}$.

\begin{Lemma} \label{L:specisos}
The map $\xi : \kk[\g^*]^{(1)} \rightarrow Z_p(\g)$ induces the vertical isomorphisms in the following commutative diagram
\begin{center}
\begin{tikzpicture}[node distance=2cm, auto]
\pgfmathsetmacro{\shift}{0.3ex}
\node (A) {$(\kk[\chi + \hat\m]^{(1)})^M$};
\node (G) [right of=A] {};
\node (B) [right of=G] {$\kk[\chi + \hat\m]^{(1)}$};
\node (H) [right of=B] {};
\node (C) [right of=H] {$\kk[\chi + \check\v]^{(1)}$};

\node (D) [below of=A] {$Z_p(\g,e)$};
\node (E) [below of=B] {$Z_p(\g)/I_p$};
\node (F) [below of=C] {$Z_p(\g)/\bar I_p$};

\draw[right hook->] (A) --(B) node[above,midway] {$ $};
\draw[->>] (B) --(C) node[above,midway] {$ $};

\draw[right hook->] (D) --(E) node[above,midway] {$ $};
\draw[->>] (E) --(F) node[above,midway] {$ $};

\draw[->] (A) --(D) node[right,midway] {$\sim$};
\draw[->] (B) --(E) node[right,midway] {$\sim$};
\draw[->] (C) --(F) node[right,midway] {$\sim$};
\end{tikzpicture}
\end{center}
and the composition of the maps in each row are isomorphisms.
\end{Lemma}

The inclusion $\phi : Z_p(\g,e) \hookrightarrow Z_p(\g)/I_p$ induces a dominant morphism of maximal spectra
\begin{equation}\label{e:specinclusion}
\phi^* : \Spec(Z_p(\g)/I_p) \rightarrow \Spec(Z_p(\g,e)),
\end{equation}
After identifying $\Spec(Z_p(\g)/I_p)$ with $(e+\m^\perp)^{(1)}$,
we can apply Lemma~\ref{L:isomorphismlemma} to deduce that $\Spec(Z_p(\g,e))$ is an orbit space for the action of $M$ on
$(e+\m^\perp)^{(1)}$.  Thus we obtain the following lemma.

\begin{Lemma}\label{L:conjugationlem}
The map (\ref{e:specinclusion}) is surjective and the fibre over any maximal ideal of $Z_p(\g,e)$ is a single $M$-orbit.
\end{Lemma}

The next theorem is central to this article.
It was obtained for $p$ sufficiently large
using reduction modulo $p$ in \cite[Theorem~2.1]{PrCQ}, see also
\cite[Remark 2.1]{PrGR} for (iii); there the role of $U(\g,e)$ was played by the modular reduction
of the ordinary finite $W$-algebra. Since $I_p$ is generated by $\xi(\m_\chi^{(1)})$ the PBW theorem for $U(\g)$
implies that the composition $Z_p(\bfp) \hookrightarrow Z_p(\g) \twoheadrightarrow Z_p(\g)/I_p$ is an
isomorphism and we identify $Z_p(\bfp)$ with $Z_p(\g)/I_p$ below.

\begin{Theorem}\label{T:extendedpcentre} $ $
\begin{enumerate}
\item[(i)] $\hU(\g,e)$ is generated by $U(\g,e)$ and $Z_p(\bfp)$.
\item[(ii)] $\hU(\g,e) \cong Z_p(\bfp) \otimes_{Z_p(\g,e)} U(\g,e) \cong Z_p(\bfa) \otimes U(\g,e)$.
\item[(iii)] $U(\g,e)$ is a free $Z_p(\g,e)$-module of rank $p^{\dim \g^\chi}$.
\item[(iv)] $\hU(\g,e)$ is a free $Z_p(\bfp)$-module of rank $p^{\dim \g^\chi}$.
\end{enumerate}
\end{Theorem}

\begin{proof}
Part (i) and the isomorphism $\hU(\g,e)  \cong Z_p(\bfa) \otimes U(\g,e)$
in (ii) follow directly from Corollary~\ref{C:UvsUhat}.
Using Theorem~\ref{T:PBWbasisthm}, we can obtain
descriptions of PBW generators of $Z_p(\g,e)$ as
\begin{equation} \label{e:pcentPBW}
\Phi(x_i) =  \xi(x_i) + \sum_{|\ba|_e = n_i + 2} \lambda_{\ba,i} \xi(x^\ba) + I_p \in Z_p(\g,e),
\end{equation}
for $i=1,\dots,r$.
Then we can
prove that the multiplication map
\begin{equation} \label{e:multpcent}
Z_p(\bfa) \otimes Z_p(\g,e) \isoto Z_p(\bfp)
\end{equation}
is an isomorphism by the same reasoning as in Corollary~\ref{C:Qfree}.

Thus we obtain isomorphisms
\begin{align*}
U(\g,e) \otimes Z_p(\bfa) &\cong (U(\g,e) \otimes_{Z_p(\g,e)} Z_p(\g,e)) \otimes Z_p(\bfa) \\
&\cong U(\g,e) \otimes_{Z_p(\g,e)} (Z_p(\g,e) \otimes Z_p(\bfa)) \\
&\cong U(\g,e) \otimes_{Z_p(\g,e)} Z_p(\bfp),
\end{align*}
which confirms (ii).

To prove (iii), we consider $\gr Z_p(\g,e)$ and $\gr U(\g,e)$ as
subalgebras of $\gr U(\g,e) \cong S(\bfp)$.
From \eqref{e:pcentPBW}, we see that
$\gr Z_p(\g,e)$ identifies with $(\gr U(\g,e))^p$, so that $\gr U(\g,e)$ is a free
$\gr Z_p(\g,e)$-module of rank $p^{\dim \g^\chi}$.
Now a standard filtration argument shows
that $U(\g,e)$ is free of rank $p^{\dim \g^\chi}$ over $Z_p(\g,e)$.

Part (iv) follows from (ii), the isomorphism \eqref{e:multpcent}, and (iii).
\end{proof}

\subsection{\texorpdfstring{The reduced finite $W$-algebras}{The reduced finite W-algebras}}
We identify $Z_p(\g,e)$ with $\kk[\chi + \check\v]^{(1)}$ via the isomorphisms $(\kk[\chi + \hat\m]^{(1)})^M \isoto
Z_p(\g,e)$ and $(\kk[\chi + \hat\m]^{(1)})^M \isoto
\kk[\chi + \check\v]^{(1)}$ from Lemma~\ref{L:specisos}.
Thus the maximal ideals of $Z_p(\g,e)$ are naturally parameterized by $\chi + \check\v$.
For $\eta \in \chi + \check\v$ we write $K_\eta$ for the corresponding maximal ideal of $Z_p(\g,e)$.

\begin{Definition}
The {\em reduced finite $W$-algebra for the $p$-character $\eta \in \chi + \check\v$} is
$$
U_\eta(\g,e) := U(\g,e) / K_\eta U(\g,e).
$$
The {\em restricted finite $W$-algebra} is $U_\chi(\g,e)$.
\end{Definition}

We remark that $\dim U_\eta(\g,e) = p^{\dim \g^\chi}$, because $U(\g,e)$ is a free module
of rank $p^{\dim \g^\chi}$ over $Z_p(\g,e)$ by Theorem~\ref{T:extendedpcentre}.

For $\eta \in \chi + \hat\m$
we define $\hK_\eta$  to be the maximal ideal of $\hZ(\g,e) = Z_p(\g) / I_p$ corresponding to $\eta$.
In the next lemma, we observe that the reduced finite $W$-algebras identify with the $p$-central reductions of $\hU(\g,e)$.
In the statement of the lemma we use the surjective map
$\phi^* : \Spec(Z_p(\g)/I_p) \twoheadrightarrow \Spec(Z_p(\g,e))$ from \eqref{e:specinclusion}, which
we view as a map
$$
\phi^* : \chi + \hat \m \twoheadrightarrow \chi + \check\v
$$
via the identifications $Z_p(\g)/I_p = \kk[\chi + \hat \m]^{(1)}$
and $Z_p(\g,e) = \kk[\chi + \check\v]^{(1)}$.

\begin{Lemma}
Let $\hK_\eta, \hK_{\eta'} \subseteq \hZ_p(\g,e)$ be maximal ideals. The following are equivalent:
\begin{enumerate}
\item $\eta$ and $\eta'$ are $M$-conjugate;
\item $\hK_\eta \cap Z_p(\g,e) = \hK_{\eta'} \cap Z_p(\g,e)$
\end{enumerate}
In this case the algebras $U_{\phi^* \eta}(\g,e)$ and $U_{\phi^* \eta'} (\g,e)$ are actually equal.
\end{Lemma}

\begin{proof}
This follows from Lemma~\ref{L:conjugationlem}.
\end{proof}

We move on to show that the definition of reduced $W$-algebras
here coincides with that given by Premet in \cite[Sections~2 and 3]{PrST};
we note that in \cite[Remark~2.1]{PrGR} an equivalent statement was observed
to be true by Premet for $p$ sufficiently large.  Before we get to this
in Proposition~\ref{P:reducesame}, we need to recall some notation.

For each $\eta \in \g^*$ we consider the maximal ideal $J_\eta$ of $Z_p(\g)$
as defined in \S\ref{ss:notation}.  We define the {\em reduced Gelfand--Graev module}
$$
Q^\eta := Q / J_\eta Q.
$$
Now it follows from Lemma~\ref{L:pcentreannQ} that $Q^\eta = 0$ if and only if $\eta \notin \chi + \hat \m$.
Further, for $\eta \in \chi + \hat\m$, the arguments in the proof of \cite[Lemma~2.2(i)]{PrCQ} can be used to show that
$$
Q^\eta \cong U_\eta(\g) \otimes_{U_\eta(\m)} \kk_\eta \cong U_\eta(\g)/I_\eta,
$$
where $I_\eta = U_\eta(\g) \m_\chi$.
There is a well defined adjoint action of $\m$ on $Q^\eta$, and
the map
\begin{align*}
\End_{U(\g)}(Q^\eta)^\op &\to (Q^\eta)^{\ad(\m)} \\
 u &\mapsto u (1 + I_\eta)
\end{align*}
is a well-defined isomorphism; this can be proved in same way as Lemma~\ref{L:invariantsareanalgebra},
using Frobenius reciprocity to construct the isomorphism.  Thus we see that $(Q^\eta)^{\ad(\m)}$ inherits an algebra
structure from $U_\eta(\g)$.

We saw in Theorem~\ref{T:PBWbasisthm}
that $U(\g,e)$ has a nice set of PBW generators $\Theta(x_1),...,\Theta(x_r)$ such
that the linear part of the highest Kazhdan degree term of $\Theta(x_i)$ is $x_i$ where $x_1,...,x_r$ is a homogeneous basis of $\g^\chi$.
We now invoke an entirely different proof from \cite[\S3.4]{PrST} which shows that $\End_{U(\g)}(Q^\eta)^\op$ has a very similar basis.
Namely there exist elements $\theta_1,...,\theta_r$ in $\End_{U(\g)}(Q^\eta)^\op$
such that
$$
\gr_{n_i+2} \theta_i = x_i + \sum_{|\ba|_e = n_i + 2} \lambda'_{\ba,i} x^\ba
$$
where $\lambda'_{\ba,i} \in \kk$, and $\lambda'_{\ba,i} = 0$ whenever $|\ba| \leq 1$.
Further, the restricted PBW monomials form a $\kk$-basis of $U_\eta(\g,e)$
\begin{equation} \label{e:reducedbasis}
\{\theta^\ba := \theta_1^{a_1}\cdots \theta_r^{a_r} \mid \ba \in \Z_{[0,p)}^r\}.
\end{equation}
We remind the reader that $\Z_{[0,p)} = \{a \in \Z \mid 0 \le a < p\}$.

By Lemma~\ref{L:invariantsareanalgebra} we know that $U(\g,e) \sub \End_{U(\g)}(Q)^\op$.
Furthermore, $J_\eta Q$ is a $U(\g,e)$-submodule of $Q$, so there is a natural map
$$
\pi_\eta : U(\g,e) \to \End_{U(\g)}(Q^\eta)^\op.
$$

\begin{Proposition} \label{P:reducesame}
The map $\pi_\eta$ is surjective with kernel $K_\eta U(\g,e)$,
so induces an isomorphism
$$
\pi_\eta : U_\eta(\g,e) \isoto \End_{U(\g)}(Q^\eta)^\op.
$$
\end{Proposition}

\begin{proof}
To see that $\pi_\eta$ surjects we observe that the basis elements from
\eqref{e:reducedbasis} are the images of the corresponding basis elements of $U(\g,e)$,
at least up to lower degree terms.
Next we note that
$K_\eta U(\g,e)$ certainly lies in the kernel of $\pi_\eta$.  To
finish the proof we just need to observe that $\dim U_\eta(\g,e) = p^{\dim \g^\chi} = \dim \End_{U(\g)}(Q^\eta)^\op$.
\end{proof}

Next we state a version of Corollary~\ref{C:Qfree} for $U_\eta(\g,e)$; this is
is proved in \cite[Lemma~2.3]{PrCQ}, and can also be deduced from Corollary~\ref{C:Qfree}.

\begin{Lemma}\label{L:globalbasis}
For $\eta \in \chi + \check\v$ the right $U_\eta(\g,e)$-module $Q^\eta$ is free of rank $p^{d_\chi}$ with basis
\begin{equation*}
\{x^\ba + I_\eta  \in Q^\eta : \ba \in \Z_{[0,p)}^m , a_1 = \cdots = a_r = 0\}.
\end{equation*}
\end{Lemma}

Our final lemma in this section says that $Q^\eta$ is faithful as a $U_\eta(\g)$-module.
This can be deduced from \cite[Lemma~2.2]{PrCQ} as in \cite[\S3.15]{To}.

\begin{Lemma} \label{L:Qetafaith}
$\Ann_{U_\eta(\g)}(Q^\eta) = \{0\}$.
\end{Lemma}

We also note that
the above lemma can be proven in an elementary way using PBW bases and commutators and some ideas
from the proof of the theorem in \cite{Sk}; we just outline this here.

We need to fix bases $\{y_1,\dots,y_{m-r}\}$ of $\bfa$ and $\{z_1,\dots,z_{m-r}\}$
of $\m$ such that $\chi([y_i,z_j]) = \delta_{ij}$.
Further, we require $y_i \in \g(s_i-2), z_i \in \g(-s_i)$, where $s_i \in \Z_{\ge 0}$ and $s_1 \ge \dots \ge s_{m-r}$.
Then $\{x_1,\dots,x_r,y_1,\dots,y_{m-r},z_1,\dots,z_{m-r}\}$ is a basis of $\g$, and we have a corresponding PBW basis
$\{x^\ba y^\bb z^\bc \mid \ba \in \Z_{[0,p)}^r, \bb \in \Z_{[0,p)}^{m-r}, \bc \in \Z_{[0,p)}^{m-r} \}$
of $U_\eta(\g)$.  We fix an order $<$ on
$\Z_{[0,p)}^{m-r}$ subject to the condition: $\ba < \bb$ whenever either $\sum_{i=1}^{m-1} s_ia_i < \sum_{i=1}^{m-r} s_ib_i$, or
$\sum_{i=1}^{m-1} s_ia_i = \sum_{i=1}^{m-r} s_ib_i$ and $\sum_{i=1}^{m-1} a_i < \sum_{i=1}^{m-r} b_i$.
We also define $z_{i,\chi} = z_i - \chi(z_i)$ for each $i = 1,\dots,m-r$ and $z_\chi^\bc = z_{1,\chi}^{c_1} \cdots z_{m-r,\chi}^{c_{m-r}}$.

The key step is to show that $z_\chi^\bc(y^\bb + I_\eta) = 0$ if $\bb < \bc$, and that
$z_\chi^\bb(y^\bb+I_\eta) = \prod_{i=1}^{m-r} b_i! + I_\eta$.  This can be achieved with an explicit calculation with
commutators.  Now given any $u = \sum_{\ba,\bb,\bc} \alpha_{\ba,\bb,\bc} x^\ba y^\bb z^\bc \in U_\eta(\g)$,
we choose $\bc$ to be maximal such that $\alpha_{\ba,\bb,\bc} \ne 0$ for some $\ba, \bb$.  Then the above formulae
can be used to prove that $u(y^{\bc} + I_\eta) \ne 0$.

Another alternative proof can be obtained by using the ideas of the proof of Theorem~\ref{T:extendedPBWtheorem}
to first show that $Q^\eta$ is free as a $U_\eta(\m)$-module and then to deduce the result.

\section{Skryabin's equivalence}
In this final section we prove a modular version of the celebrated category equivalence due to Skryabin
in characteristic zero, see \cite{Sk}. We build on the approach of the second author \cite{To} where
this result was first obtained for $p$ sufficiently large using Premet's modular reduction of the
characteristic zero finite $W$-algebra.

\subsection{Central reductions of enveloping algebras}
Let $\h$ be a subspace of $\g$.  We define $\h_\chi = \{x - \chi(x) \mid x \in \h^\perp\}$ and $J_{\chi,\h}$ to be the ideal of $Z_p(\g)$
generated by $\xi(\h^{(1)}_\chi) = \{x^p - x^{[p]} - \chi(x)^p \mid x \in \h^\perp\}$.  Then we define the central reduction
$$
U_{\chi,\h}(\g) := U(\g) / U(\g)J_{\chi,\h}
$$
of $U(\g)$.
We are interested in $U_{\chi,\m^\perp}(\g)$ and $U_{\chi,\v}(\g)$. Their relationship is described in the next lemma.

\begin{Lemma}\label{L:surjkernel}
There is a natural surjection $U_{\chi,\m^\perp}(\g) \onto U_{\chi,\v}(\g)$ with kernel $\xi(\bfa^{(1)})U_{\chi,\m^\perp}(\g)$.
\end{Lemma}

\begin{proof}
We know that the restriction of the form $(\cdot\,,\cdot)$ to $\m$ is zero and
that $\v \subseteq \m$.  Thus it follows that $\m \sub \v^\perp = 0$. This shows that the surjection exists and that the kernel is
generated by $\{x^p - x^{[p]} - \chi(x)^p \mid x\in \bfp \cap \v^\perp\}$. The claim now follows since $\bfa = \{x \in \bfp \mid (x, \v) = 0\}$.
\end{proof}

We note that $U_{\chi,\m^\perp}(\g)$ is naturally a $Z_p(\bfp)$-module, because
$J_{\chi,\m^\perp} = I_p$ and the map $Z_p(\bfp) \to Z_p(\g)/I_p$ is an isomorphism; here
$I_p$ is as defined at the start of \S\ref{ss:pcentre}.
Also Lemma~\ref{L:pcentreannQ} states that
\begin{equation}\label{e:restatepaQ}
I_p = \Ann_{Z_p(\g)}(Q),
\end{equation}
so $Q$ and $\hU(\g,e)$ are $Z_p(\bfp)$-modules in a natural way, and $\Mat_{p^{d_\chi}} \hU(\g,e)$
is also a $Z_p(\bfp)$-module.
We write $Q^{\oplus p^{d_\chi}}$ for
the direct sum of $p^{d_\chi}$ copies of $Q$.

Now we can state and prove the following theorem, which forms
the main step in obtaining our version of Skryabin's equivalence in
Theorem~\ref{T:skryabin}

\begin{Theorem} \label{T:matrixalgebratheorem}
$ $
\begin{enumerate}
\item[(i)] $Q$ is a free right $\hU(\g,e)$-module of rank $p^{d_\chi}$.
\item[(ii)] There exists an isomorphism of left $U(\g)$-modules
$$
U_{\chi,\m^\perp}(\g) \isoto Q^{\oplus p^{d_\chi}};
$$
\item[(iii)] $Q$ is a projective generator for $U_{\chi,\m^\perp}(\g)\lmod$.
\item[(iv)] There exists a $Z_p(\bfp)$-equivariant isomorphism of algebras
$$
U_{\chi,\m^\perp}(\g) \isoto \Mat_{p^{d_\chi}} \hU(\g,e);
$$
\item[(v)] There exists an isomorphism of algebras
$$
U_{\chi,\v}(\g) \isoto \Mat_{p^{d_\chi}} U(\g,e).
$$
\end{enumerate}
\end{Theorem}

\begin{proof}
We start by noting that (i) can be deduced as a
consequence of Corollary~\ref{C:Qfree} and Theorem~\ref{T:extendedpcentre}(ii).

The proof of (ii), (iii) and (iv) follow similar steps as in \cite[Section~5]{To}, however
we include the argument for the readers convenience, and so that the differences can be seen.
Consider the set
$$
\{x^\ba +I \in Q : \ba \in \Z_{[0,p)}^m, a_1 = \cdots = a_r = 0\}
$$
from Lemma~\ref{L:globalbasis}. We label this set $v_1,...,v_{p^{d_\chi}} \in Q$ in some arbitrary manner.
For $i=1,...,p^{d_\chi}$ we let $Q[i]$ be an isomorphic copy of $Q$ and
we write $v_i[i] \in Q[i]$ for the element of $Q[i]$ corresponding to $v_i$.
Next we define a $\g$-module homomorphism $\varphi : U(\g) \to \bigoplus_{i=1}^{p^{d_\chi}} Q[i] = Q^{\oplus p^{d_\chi}}$ by
\begin{equation} \label{e:varphiso}
\varphi(u) = \sum_{i=1}^{p^{d_\chi}} u v_i[i]
\end{equation}
For $\eta \in \g^*$ write $\varphi_\eta : U_\eta(\g) \to (Q^\eta)^{\oplus p^{d_\chi}}$
for the induced map on the quotients. It follows from Lemma~\ref{L:Qetafaith} that $\varphi_\eta$ is injective
for $\eta \in \chi + \hat \m$.  From \eqref{e:dimm+dimn} we have $\dim \m = d_\chi$ and so $\dim (Q^\eta)^{\oplus p^{d_\chi}} = p^{\dim \g} = \dim U_\eta(\g)$
for all such $\eta$. It follows that $\varphi_\eta$ is also surjective.
By \eqref{e:restatepaQ} we have $Q^\eta = 0$ for $\eta \notin \chi + \hat \m$ so actually
$\varphi_\eta$ is surjective for every $\eta \in \g^*$, thus
by \cite[Corollary~4.4]{To} we have
$\coker(\varphi)/J_\eta \coker(\varphi) = \coker(\varphi_\eta) = 0$.
Hence, $\coker(\varphi) = 0$ by \cite[Lemma~4.4]{To}, so that $\varphi$ is surjective.

Using Lemma~\ref{L:pcentreannQ}  we have a
surjection $U_{\chi,\m^\perp}(\g) \onto Q^{\oplus p^{d_\chi}}$.
Also $U_{\chi,\m^\perp}(\g)$ is a free $Z_p(\bfp)$-module of rank $p^{\dim(\g)}$, as follows directly from the fact
that $U(\g)$ is a free $Z_p(\g)$-module of rank $p^{\dim \g}$.  We have that
$Q$ a free $Z_p(\bfp)$-module of rank $p^{\dim \bfp}$, so that $Q^{\oplus p^{d_\chi}}$ is free of rank
$p^{\dim \g}$.  Thus $U_{\chi,\m^\perp}(\g)$ and $Q^{\oplus p^{d_\chi}}$ are isomorphic as finitely generated
$Z_p(\bfp)$-modules, so that $U_{\chi,\m^\perp}(\g) \onto Q^{\oplus p^{d_\chi}}$ must be an isomorphism, see
for example \cite[Theorem~2.4]{Ma}. Thus we have proved (ii).

Now (iii) follows from the characterization of projective generators, see for example \cite[\S18B]{La}.

From (ii), standard arguments prove that
$$
U_{\chi,\m^\perp}(\g) \cong \End_{U(\g)}(Q^{\oplus p^{d_\chi}})^\op \cong (\Mat_{p^{d_\chi}} \End_{U(\g)}(Q))^\op \cong \Mat_{p^{d_\chi}} \hU(\g,e),
$$
which gives (iv).

To prove (v), we recall that $\xi(\bfa^{(1)}) := \{x^p - x^{[p]} \mid x\in \bfp, (x, \v) = 0\} \sub Z_p(\bfp)$, and then
using Lemma~\ref{L:surjkernel}, part (iv) of the current theorem, and Corollary~\ref{C:UvsUhat}, we obtain
\begin{align*}
U_{\chi,\v}(\g)  &\cong U_{\chi,\m^\perp}(\g)  / \xi(\bfa^{(1)}) U_{\chi,\m^\perp}(\g)  \\
& \cong  \Mat_{p^{d_\chi}} \hU(\g,e) / \xi(\bfa^{(1)}) \Mat_{p^{d_\chi}} \hU(\g,e)\\
&\cong \Mat_{p^{d_\chi}} (\hU(\g,e) / \xi(\bfa^{(1)}) \hU(\g,e)) \\ &\cong  \Mat_{p^{d_\chi}} U(\g,e). & & \qedhere
\end{align*}
\end{proof}

\subsection{\texorpdfstring{Skryabin's equivalence for modular finite $W$-algebras}
{Skryabin's equivalence for modular finite W-algebras}}
We move on to state and prove our modular analogue of Skryabin's equivalence after giving
the notation required.
From the proof of Theorem~\ref{T:matrixalgebratheorem}
it follows that we have a commutative diagram of algebra homomorphisms
 \begin{center}
\begin{tikzpicture}[node distance=2cm, auto]
\pgfmathsetmacro{\shift}{0.3ex}
\node (G) {$U(\g,e)$};
\node (H) [right of=G] {};
\node (I) [below of = G] {$\hU(\g,e)$};
\node (A) [right of=H] {$\Mat_{p^{d_\chi}} U(\g,e)$};
\node (E) [right of=A] { };
\node (B)[right of=E] {$U_{\chi,\v}(\g)$};
\node (C) [below of=A] {$\Mat_{p^{d_\chi}} \hU(\g,e)$};
\node (F) [right of=C] { };
\node (D) [right of=F] {$U_{\chi,\m^\perp}(\g)$};

\draw[transform canvas={yshift=0.5ex},right hook->] (G) --(A) node[above,midway] {$ $};

\draw[transform canvas={yshift=0.5ex},right hook->] (I) --(C) node[above,midway] {$ $};

\draw[transform canvas={yshift=0.5ex},->] (A) --(B) node[above,midway] {$\sim$};

\draw[transform canvas={yshift=0.5ex},->] (C) --(D) node[above,midway] {$\sim$};

\draw[transform canvas={xshift=0.5ex},->>] (I) --(G) node[right,midway] {$ $};

\draw[transform canvas={xshift=0.5ex},->>] (C) --(A) node[right,midway] {$ $};

\draw[transform canvas={xshift=0.5ex},->>] (D) --(B) node[right,midway] {$ $};
\end{tikzpicture}
\end{center}
with the $W$-algebras embedded diagonally in the matrix algebras.
The vertical surjections on the right hand side of this diagram induce functors
\begin{center}
\begin{tikzpicture}[node distance=3.5cm, auto]
\pgfmathsetmacro{\shift}{0.3ex}
\node (A) {$U_{\chi,\v}(\g)\lmod$};
\node (B) [right of=A] {$U_{\chi,\m^\perp}(\g)\lmod$,};
\draw[transform canvas={yshift=0.5ex},->] (A) --(B) node[above,midway] {$\scriptstyle j$};
\draw[transform canvas={yshift=-0.5ex},->](B) -- (A) node[below,midway] {$\scriptstyle \rho$};
\end{tikzpicture}
\end{center}
where $j$ is the inclusion functor given by the pullback through the surjection
$U_{\chi,\m^\perp}(\g) \onto U_{\chi,\v}(\g)$, and $\rho : V \mapsto V/ \xi(\bfa^{(1)})V$.
Moreover, $(\rho,j)$ is an adjoint pair.
Similarly the left hand surjection induces an adjoint pair $(\pi,i)$  of functors
\begin{center}
\begin{tikzpicture}[node distance=3.5cm, auto]
\pgfmathsetmacro{\shift}{0.3ex}
\node (A) {$U(\g,e)\lmod$};
\node (B) [right of=A] {$\hU(\g,e)\lmod$.};
\draw[transform canvas={yshift=0.5ex},->] (A) --(B) node[above,midway] {$\scriptstyle i$};
\draw[transform canvas={yshift=-0.5ex},->](B) -- (A) node[below,midway] {$\scriptstyle \pi$};
\end{tikzpicture}
\end{center}
We observe that if $V$ is a $U(\g)$-module then $\hU(\g,e)$ acts
naturally on
$$
V^{\m_{\chi}} := \{v \in V \mid (x - \chi(x))v = 0 \text{ for } x \in \m\}.
$$
Also given a $\hU(\g,e)$-module $W$, we have that
$$
Q \otimes_{\hU(\g,e)} W
$$
is a $U(\g)$-module.

We have now set up all the notation required to state our final theorem.

\begin{Theorem} \label{T:skryabin}
The restriction of $(-)^{\m_{\chi}}$ to $U_{\chi,\v}(\g)\lmod$ has image in
$U(\g,e)\lmod$ and the restriction of $Q \otimes_{\hU(\g,e)} (-)$ to $U(\g,e)\lmod$ has image in
$U_{\chi,\v}(\g)\lmod$.
Therefore, we obtain the following diagram of functors
\begin{center}
\begin{tikzpicture}[node distance=3.5cm, auto]
\pgfmathsetmacro{\shift}{0.3ex}

\node (A) {$U(\g,e)\lmod$};
\node (E) [right of=A] { };
\node (B)[right of=E] {$U_{\chi,\v}(\g)\lmod$};
\node (C) [below of=A] {$\hU(\g,e)\lmod$};
\node (F) [right of=C] { };
\node (D) [right of=F] {$U_{\chi,\m^\perp}(\g)\lmod$.};

\draw[transform canvas={yshift=0.5ex},->] (A) --(B) node[above,midway] {$Q \otimes_{\hU(\g,e)} (-)$};
\draw[transform canvas={yshift=-0.5ex},->](B) -- (A) node[below,midway] {$(-)^{\m_{\chi}}$};

\draw[transform canvas={yshift=0.5ex},->] (C) --(D) node[above,midway] {$Q \otimes_{\hU(\g,e)} (-)$};
\draw[transform canvas={yshift=-0.5ex},->](D) -- (C) node[below,midway] {$(-)^{\m_{\chi}}$};

\draw[transform canvas={xshift=0.5ex},->] (A) --(C) node[right,midway] {$i$};
\draw[transform canvas={xshift=-0.5ex},->](C) -- (A) node[left,midway] {$\pi$};

\draw[transform canvas={xshift=0.5ex},->] (B) --(D) node[right,midway] {$j$};
\draw[transform canvas={xshift=-0.5ex},->](D) -- (B) node[left,midway] {$\rho$};
\end{tikzpicture}
\end{center}
Moreover, the horizontal arrows at both the top and bottom are inverse equivalences of categories.
\end{Theorem}

\begin{proof}
First we deal with the claims regarding the bottom row of the diagram, which essentially follow from Morita theory,
see for example \cite[\S18]{La}.
By Theorem~\ref{T:matrixalgebratheorem}(iii) we know that $Q$ is a projective generator for $U_{\chi,\m^\perp}(\g)$,
and its endomorphism algebra is $\hU(\g,e)$.
Thus using \cite[Theorem~18.24]{La}, we obtain that $Q \otimes_{\hU(\g,e)} (-) : \hU(\g,e)\lmod \to U_{\chi,\m^\perp}(\g)\lmod$
is an equivalence of categories and then also using \cite[Remark~18.25]{La} an inverse equivalence is given by
$\Hom_{U_{\chi,\m^\perp}(\g)}(Q,-) :  U_{\chi,\m^\perp}(\g)\lmod \to \hU(\g,e)\lmod$.  Let $V$ be a $U_{\chi,\m^\perp}(\g)$-module.
Then we have $\Hom_{U_{\chi,\m^\perp}(\g)}(Q,V) = \Hom_{U(\g)}(Q,V)$
and observe that
$$
\Hom_{U(\g)}(Q,V) \cong \Hom_{U(\m)}(\kk_\chi,V) \cong V^{\m_\chi},
$$
where we use Frobenius reciprocity for the first isomorphism.  This
shows that the horizontal arrows at the bottom of the diagram are inverse equivalences.

Now let $V \in U_{\chi,\v}(\g))\lmod$ and  $W \in U(\g,e)\lmod$.
Viewing $V$ as a module for $U_{\chi,\v}(\g)$ via $j$, we note that $\xi(\bfa)(W^{\m_\chi}) = 0$, so that $W^{\m_\chi} \in U(\g,e)\lmod$.
Similarly, viewing $W$ as a module for $\hU(\g,e)$ via $i$ we observe that $\xi(\bfa)(Q \otimes_{\hU(\g,e)} W) = Q \otimes_{\hU(\g,e)} \xi(\bfa)W = 0$.
Therefore, we have that $Q \otimes_{\hU(\g,e)} W \in U_{\chi,\v}(\g)\lmod$.
It now follows that the horizontal arrows at the top of the diagram are inverse equivalences.
\end{proof}

\begin{Remark} \label{R:genequiv}
We end by noting that the proof Theorem~\ref{T:skryabin} can be adapted to consider other central quotients of
$\hU(\g,e)$ and $U_{\chi,\m^\perp}(\g)$.  More specifically, given any ideal $J$ of $Z_p(\g)$ containing $J_{\chi,\m^\perp}$,
we can consider $U_{\chi,\m^\perp}(\g)/JU_{\chi,\m^\perp}(\g)$ and $\hU(\g,e)/\hJ\hU(\g,e)$, where $\hJ$ denotes the
image of $J$ in $\hU(\g,e)$.  Then the proof of Theorem~\ref{T:skryabin}
can be used to give an equivalence of categories between $U_{\chi,\m^\perp}(\g)/JU_{\chi,\h}(\g)\lmod$ and $\hU(\g,e)/\hJ\hU(\g,e)\lmod$
via the functors  $(-)^{\m_{\chi}}$ and $Q \otimes_{\hU(\g,e)} (-)$.

In the special case $J = J_\eta$ with $\eta \in \chi + \hat \m$, we obtain an
equivalence between the category of modules for
$U_{\chi,\m^\perp}(\g)/J_\eta U_{\chi,\h}(\g) \cong U_\eta(\g)$ and the category of
modules for $\hU(\g,e)/\hJ_\eta\hU(\g,e) \cong U_\eta(\g,e)$.
This recovers Premet's equivalence from \cite[Theorem~2.4]{PrST} (see also \cite[Lemma~2.1]{PrCQ}),
and we note that by using the (elementary) proof of Lemma~\ref{L:Qetafaith}
outlined after its statement, we would obtain an alternative proof for this equivalence.
\end{Remark}

\end{document}